\documentclass[a4paper]{amsart} 

\usepackage{adjustbox}

\usepackage{amsmath}
\usepackage{amsthm}
\usepackage{amssymb} 
\usepackage[alphabetic]{amsrefs}

\usepackage{mathtools}

\usepackage{amsrefs} 
\usepackage[colorlinks=true, linkcolor=black, anchorcolor=black,citecolor=black,filecolor=black,menucolor=black,runcolor=black,urlcolor=black]{hyperref}

\usepackage{cleveref} 

\usepackage{color} 

\usepackage{booktabs} 
\usepackage{multicol} 
\usepackage{multirow} 
\usepackage{rotating} 

\theoremstyle{plain} 
\newtheorem{theorem}{Theorem}[section]
\newtheorem*{theorem*}{Theorem}
\newtheorem{lemma}[theorem]{Lemma}
\newtheorem{proposition}[theorem]{Proposition}

\newtheorem{conjecture}[theorem]{Conjecture}

\newtheorem{maintheorem}{Theorem}[section]

\theoremstyle{definition} 
\newtheorem{definition}[theorem]{Definition}

\theoremstyle{remark} 

\newcommand{\R}{\mathbb{R}}
\newcommand{\Z}{\mathbb{Z}}
\newcommand{\C}{\mathbb{C}}
\newcommand{\Sph}{\mathbb{S}}

\newcommand{\Quaternion}{\mathbb{H}}
\newcommand{\F}{\mathbb{F}}

\DeclareMathOperator{\Span}{span}
\newcommand{\de}{\mathrm{d}}
\newcommand{\deint}{\;\mathrm{d}}
\DeclareMathOperator{\trace}{tr}
\DeclareMathOperator{\ad}{ad}
\DeclareMathOperator{\Hom}{Hom}

\newcommand{\diff}[2]{\dfrac{\de #1}{\de #2}}
\newcommand{\lb}[2]{\left[ #1,#2 \right]}
\renewcommand{\iff}{\Leftrightarrow}
\newcommand{\iprod}[2]{\left\langle #1, #2 \right\rangle}

\newcommand{\prep}[2]{\mathfrak{p}_{#1}^{(#2)}}
\newcommand{\qrep}[2]{\mathfrak{q}_{#1}^{(#2)}}

\newcommand{\quotient}[2]{#1/#2}

\allowdisplaybreaks

\title{The Alekseevskii Conjecture in 9 and 10 Dimensions}
\author{Rohin Berichon}
\address{School of Mathematics and Physics, The University of Queensland, St Lucia, QLD, 4072, Australia}
\email{r.berichon@uq.edu.au}

\begin{document}
    \maketitle
    \begin{abstract}
        We show that non-compact homogeneous spaces not diffeomorphic to Euclidean space of dimension 9 or 10 admit no homogeneous Einstein metrics of negative Ricci curvature, with only three potential exceptions. The main ingredient in the proof is to show, via a cohomogeneity-one approach, that non-compact homogeneous spaces admitting an ideal isomorphic to $ \mathfrak{sl}_2(\R) $ admit no homogeneous Einstein metrics of negative Ricci curvature.
    \end{abstract}

    \section{Introduction}
    { 
        A Riemannian manifold $ \left( \mathcal{M}, g \right) $ is called \textit{Einstein} if the Ricci tensor satisfies $ \mathrm{Ric}(g) = \lambda \cdot g $ for some $ \lambda \in \R $. In general, it is far too optimistic to provide existence criteria for solutions to the Einstein equation. Instead, a widespread method is to impose some sort of symmetry assumption, or to restrict to metrics with special holonomy. In this article, we study the existence of Einstein metrics on non-compact homogeneous spaces.
    }

    { 
        The study of homogeneous Einstein metrics is roughly grouped into the cases when the Ricci curvature is positive, zero, or negative. If the Ricci curvature is positive, the manifold is compact, and has finite fundamental group by the Bonnet-Myers theorem \cite{myers_1941}. It was proven in \cite{alekseevskii_kimelfeld_1975} that homogeneous Ricci flat manifolds are flat. Finally, it is known that homogeneous Einstein manifolds of negative Ricci curvature are non-compact, by \cite{bochner_1948}. At present, all known examples of these spaces are isometric to simply-connected solvmanifolds. We have the following

    }
    \begin{conjecture}[Alekseevskii, 1975, \cite{besse_2008}]
        Any connected homogeneous Einstein manifold of negative scalar curvature is diffeomorphic to a Euclidean space.
    \end{conjecture}

    Current results verify the Alekseevskii conjecture in dimension at most 8 with 3 possible exceptions (see \cite{arroyo_lafuente_2016} and references therein). In 2 and 3 dimensions, homogeneous Einstein manifolds have constant sectional curvature, so are diffeomorphic to $ \R^n $. \cite{jensen_1969} classified the 4 dimensional simply-connected homogeneous Einstein spaces of negative curvature, implying the conjecture. \cite{nikonorov_2005} verified the conjecture in 5 dimensions with the possible exception of spaces admitting a transitive $ \mathsf{SL}_2(\C) $ action (see Proposition \ref{cor:no_sl2c_minimals}). In the 6 dimensional case, the conjecture was verified with the exception of $ \mathsf{SL}_{2}(\C) $ and the universal cover of $ \mathsf{SL}_{2}(\R)^{2} $ by \citelist{\cite{arroyo_lafuente_2015}\cite{jablonski_petersen_2017}}, the latter of which being resolved by \cite{bohm_lafuente_2019}*{Corollary 6.4}. Finally, \cite{arroyo_lafuente_2016}*{Theorem B, C} verified the conjecture in dimensions 7 and 8 with the possible exception of $ \mathsf{SU}(2, 1) $ and $ \mathsf{SL}_3(\R) $.

    The primary goal of this paper is to prove the following

    \begin{maintheorem}\label{thm:main_alek_theorem}
        Let $ \left( \mathcal{M}^n, g \right) $ be a simply-connected homogeneous Einstein manifold with $ \lambda < 0 $ of dimension at most 10, which is de Rham irreducible. If $ \left( \mathcal{M}, g \right) $ is not an invariant metric on the universal covers of $ \mathsf{Sp}(1, 1)/\Delta\mathsf{U}(1) $, $ \mathsf{Sp}(1, 1) $, $ \mathsf{Sp}(2, \R) $, $ \mathsf{SL}_3(\R) $, $ \mathsf{SU}(2, 1) $, or $ \mathsf{SL}_2(\C) $, then $ \mathcal{M}^n $ is diffeomorphic to Euclidean space.
    \end{maintheorem}

    Regarding the simply-connected assumption, suppose $ \left( \mathcal{M}^n, g \right) $ is a homogeneous Einstein metric of negative Ricci curvature. If the universal cover $ \widetilde{\mathcal{M}} $, which is also homogeneous and Einstein, is diffeomorphic to Euclidean space, then by \cite{bohm_lafuente_2019}, $ \widetilde{\mathcal{M}} $ is isometric to a simply-connected Einstein solvmanifold. Therefore, by \cite{jablonski_2015}*{Theorem 1.1}, $ \widetilde{\mathcal{M}} $ does not admit any non-trivial quotient, thus in order to verify the Alekseevskii conjecture, it is sufficient to check only the simply-connected homogeneous spaces. 

    It is also important to note the de Rham irreducible assumption in the statement of the theorem. Assuming the Alekseevskii conjecture holds in dimensions at most 8, Theorem \ref{thm:main_alek_theorem} verifies the conjecture in dimensions 9 and 10, with the possible exceptions detailed therein, since on de Rham reducible manifolds, the Einstein equation splits into the Einstein equation on each of the irreducible factors.

    Unfortunately, our methods do not allow us to extend these results to higher dimensions. In 11 dimensions for example, the Einstein equation on the homogeneous space $ \left( \mathsf{SL}_2(\C)\cdot \R \right)\ltimes \C^2 $, with $ \mathsf{SL}_2(\C) $ acting with the standard representation on $ \C^2 $ reduces to an equation for left-invariant metrics on $ \mathsf{SL}_2(\C) $ \cite{arroyo_lafuente_2016}.

    Our main tool for proving Theorem \ref{thm:main_alek_theorem} is the following result, which allows us to ignore homogeneous spaces $ \mathsf{G}/\mathsf{H} $ with an ideal in $ \mathfrak{g} $ isomorphic to $ \mathfrak{sl}_2(\R) $.
    
    \begin{maintheorem}\label{thm:sl2r_theorem}
        If $ \left( \mathcal{M}^n, g \right) $ is a simply-connected homogeneous Einstein manifold with minimal presentation $ \mathsf{G}/\mathsf{H} $, with $ \mathsf{G} $ semisimple, with $ \mathrm{Ric}(g) = -g $, then no ideal in $ \mathfrak{g} = \mathrm{Lie}(\mathsf{G}) $ is isomorphic to $ \mathfrak{sl}_2\left( \R \right) $.
    \end{maintheorem}

    {

        We may also restrict ourselves to studying only the semisimple homogeneous spaces by \cite{dotti_1988}*{Theorem 2}, since the case when $ \mathsf{G} $ is nonunimodular was already verified in \cite{arroyo_lafuente_2016}*{Theorem D}. Moreover, irreducible symmetric spaces are all diffeomorphic to solvmanifolds \cite{helgason_1978}.

        The proof of Theorem \ref{thm:main_alek_theorem} is broken into two parts. The first is a classification, which provides a complete account for the remaining AU-spaces (see Definition \ref{def:auspace}) and their corresponding compact dual AU*-spaces. In this part, we provide the decompositions of the isotropy representations of each AU-space into irreducible submodules, indicating isomorphisms between them. In the second part, we verify the non-existence of invariant Einstein metrics of negative curvature in the remaining AU-spaces from the first part which are not covered by other results (see \citelist{\cite{arroyo_lafuente_2016}\cite{nikonorov_2000}}).

        The proof of Theorem \ref{thm:sl2r_theorem} proceeds by contradiction, assuming $ \mathfrak{g} $ admits an ideal isomorphic to $ \mathfrak{sl}_2(\R) $. We show that such a homogeneous space admits an effective cohomogeneity-one action of a closed subgroup of $ \mathsf{G} $ satisfying the conditions of \cite{bohm_lafuente_2019}*{Theorem D}. Using these results, we show that $ (\mathcal{M}^n, g) $ is locally isometric to a Riemannian product of Einstein metrics with one factor isometric to $ \mathsf{PSL}_2(\R) $, giving us our contradiction.
    }

    {
        The existence of homogeneous Einstein metrics on compact homogeneous manifolds is extensively investigated in low dimensions in \cite{bohm_kerr_2004}, with several other existence results in \citelist{\cite{dickinson_kerr_2008}\cite{wang_1982}\cite{wang_ziller_1986}\cite{yan_chen_deng_2019}}. In the case of non-compact homogeneous manifolds, structural results are known for solvmanifolds \citelist{\cite{heber_1998}\cite{lauret_2009}\cite{lauret_2010}} and more generally \citelist{\cite{arroyo_lafuente_2016}\cite{bohm_lafuente_2019}\cite{jablonski_petersen_2017}}.
    }

    {
        In \S2, we provide an overview of the necessary theory of homogeneous manifolds required for the proof of Theorems \ref{thm:main_alek_theorem} and \ref{thm:sl2r_theorem}. In \S3, we prove Theorem \ref{thm:sl2r_theorem}, and in \S4, we prove Theorem \ref{thm:main_alek_theorem}.
    }

    \section{Preliminaries}
    {
        In what follows, we will cover some of the well-known theory of homogeneous Einstein manifolds necessary for the proof of the main theorems. Throughout, we assume that all manifolds are connected, and all presentations of homogeneous spaces are almost-effective with connected transitive group and isotropy. 
    }

    {
        Let $ \mathsf{G}/\mathsf{H} $ be a homogeneous space with $ \mathsf{H} $ compact and fix a reductive decomposition $ \mathfrak{g} = \mathfrak{h}\oplus\mathfrak{m} $ for $ \mathsf{G}/\mathsf{H} $. Then, it is well known that $ \mathsf{G}/\mathsf{H} $ admits $ \mathsf{G} $-invariant Riemannian metrics. Moreover, we have the isomorphism \cite{kobayashi_nomizu_vol_II}*{Chapter X, Proposition 3.1}
        \[
            \mathcal{M}^{\mathsf{G}} \coloneqq \left\{ \begin{matrix}
                \mathsf{G}\text{-invariant metrics}\\
                \text{on }\mathsf{G}/\mathsf{H}
            \end{matrix} \right\}
            \leftrightsquigarrow
            \left\{ \begin{matrix}
                \mathrm{Ad}(\mathsf{H})\text{-invariant inner}\\
                \text{products on }\mathfrak{m}
            \end{matrix} \right\}.
        \]
        On the right is the set of positive-definite, symmetric, non-degenerate, $ \mathsf{Ad}(\mathsf{H}) $-endomorphisms on $ \mathfrak{m} $. Decompose $ \mathfrak{m} $ into $ \mathsf{Ad}(\mathsf{H}) $-irreducible submodules $ \mathfrak{m} = \mathfrak{m}_1 \oplus \mathfrak{m}_2 \oplus \cdots \oplus \mathfrak{m}_k $, and let $ Q $ be an $ \mathrm{Ad}(\mathsf{H}) $-homomorphism defining an invariant inner product on $ \mathfrak{m} $. Denote by $ Q_{ij} $ the restriction of $ Q $ to $ \mathfrak{m}_i $ projected onto $ \mathfrak{m}_j $ with respect to the decomposition above. Then, $ Q_{ij} $ is an $ \mathsf{Ad}(\mathsf{H}) $-homomorphism between irreducible modules. If $ \mathfrak{m}_i \not\simeq \mathfrak{m}_j $, then $ Q_{ij} = 0 $. On the other hand, if $ \mathfrak{m}_i \simeq \mathfrak{m}_j $, then by Schur's Lemma, $ Q_{ij} $ is either an isomorphism, or zero. Moreover, by the Frobenius Theorem, $ \mathrm{End}_{\mathrm{Ad}(\mathsf{H})}(\mathfrak{m}_i) $ is isomorphic to one of $ \R, \C $, or $ \Quaternion $. We say that $ \mathfrak{m}_i $ is of \textit{real}, \textit{complex}, or \textit{quaternionic} type, respectively. 

        Without loss of generality, assume that $ \mathfrak{m} $ is decomposed into isotypical summands $ \mathfrak{m}_1^{n_1}\oplus \cdots \oplus \mathfrak{m}_\ell^{n_\ell} $. Then, summarising the above, the space of $ \mathsf{G} $-invariant metrics on $ \mathsf{G}/\mathsf{H} $ is given by \cite{bohm_2004}
        \[
        	\mathcal{M}^{\mathsf{G}} = \mathfrak{h}^{+}_{n_1}(\F_1)\times \cdots \times \mathfrak{h}^{+}_{n_\ell}(\F_{\ell}),\quad n_1+ \cdots + n_\ell = k,
        \]
        where each $ \mathfrak{h}^{+}_{n_i}(\F_i) $ is the subspace of $ \mathfrak{gl}_{n_i}(\F_i) $ consisting of symmetric\footnote{with respect to the ground field $ \F_i $. That is, if $ \F_i = \R, \C, $ or $ \Quaternion $, then $ \mathfrak{h}_{n_i}^+(\F_i) $ consists of symmetric, hermitian, or quaternionic hermitian matrices respectively.}, positive definite, non-degenerate matrices with entries in $ \F_i \coloneqq \mathrm{End}_\mathsf{H}(\mathfrak{m}_{i}) $.
    }

    {
        Suppose $ \mathsf{G} $ is a non-compact semisimple Lie group, and $ \mathsf{H} $ is a compact subgroup contained strictly within a maximal compact subgroup $ \mathsf{K} $ of $ \mathsf{G} $. Recall here that in passing to a covering of $ \mathsf{G} $, the corresponding covering of $ \mathsf{K} $ may be non-compact. For example, there are no nontrivial compact subalgebras inside the Lie algebra of the universal covering group $ \widetilde{\mathsf{SL}_2(\R)} $ of $ \mathsf{SL}_2(\R) $, but $ \mathsf{SO}(2) $ is maximally compact in $ \mathsf{SL}_2(\R) $. Of course, at the Lie algebra level, there is no difference. Writing $ \mathfrak{g} = \mathfrak{k}\oplus\mathfrak{p} = \mathfrak{h}\oplus\mathfrak{q}\oplus\mathfrak{p} $, with the first equality giving the Cartan decomposition of $ \mathfrak{g} $ and where $ \mathfrak{q} $ is the orthogonal complement of $ \mathfrak{h} $ in $ \mathfrak{k} $ with respect to the Killing form, we obtain a decomposition of the reductive complement of $ \mathfrak{h} $ into $ \mathrm{Ad}(\mathsf{H}) $-submodules $ \mathfrak{m} = \mathfrak{q}\oplus\mathfrak{p} $. Suppose
        \[
        	\mathfrak{q} = \mathfrak{q}_1^{(n_1)}\oplus\cdots\oplus\mathfrak{q}_k^{(n_k)},\quad \mathfrak{p} = \mathfrak{p}_1^{(m_1)}\oplus\cdots\oplus\mathfrak{p}_\ell^{(m_\ell)}
        \]
        is a decomposition of $ \mathfrak{q} $ and $ \mathfrak{p} $ into $ \mathrm{Ad}(\mathsf{H}) $-irreducible submodules $ \mathfrak{q}_i^{(n_i)} $ and $ \mathfrak{p}_j^{(m_j)} $, where $ n_i = \dim\mathfrak{q}_i^{(n_i)} $ and $ m_j = \dim\mathfrak{p}_j^{(m_j)} $. Then,

        \begin{theorem}[\cite{nikonorov_2000}*{Theorem 1}]\label{thm:nikonorov_thm_1}
            Under the above assumptions, suppose $ g $ is a $ \mathsf{G} $-invariant metric on $ \mathsf{G}/\mathsf{H} $ such that $ g(\mathfrak{q}, \mathfrak{p}) = 0 $. Then, $ g $ is not Einstein.
        \end{theorem}

        By Schur's Lemma, if $ \qrep{i}{n_i} \not\simeq \prep{j}{m_j} $ for any $ i, j $, then $ \mathfrak{p} $ and $ \mathfrak{q} $ are orthogonal for every $ \mathsf{G} $-invariant metric, and hence none of them can be Einstein. On the other hand, if there are $ \qrep{i}{n_i} \simeq \prep{j}{m_j} $, then all we can gather from the decomposition of the isotropy representation into irreducible summands is a parameterisation of the space of invariant metrics in the sense of the prequel. 
    }

    {
        \begin{definition}
            We call a homogeneous space $ \mathsf{G}/\mathsf{H} $ \textit{minimally presented} if $ \mathsf{H} $ is compact, and $ \dim \mathsf{G} $ is minimal amongst all presentations of $ \mathsf{G}/\mathsf{H} $ with compact isotropy. 
        \end{definition}

        In what follows, $ \mathsf{G}/\mathsf{H} $ is a homogeneous space in minimal presentation. In \cite{heber_1998}, Heber attained structural results on so-called \textit{standard} Einstein solvmanifolds, providing criteria for uniqueness results, which were then developed on by \cite{lauret_2010} by showing all Einstein solvmanifolds are standard. Motivated by these results, \cite{bohm_lafuente_2019} generalises the \textit{standardness} condition to arbitrary homogeneous spaces in the following

        \begin{definition}[\cite{bohm_lafuente_2019}]\label{def:standard}
            Let $ \left( \mathsf{G}/\mathsf{H}, g \right) $ be a homogeneous space with compact isotropy and canonical reductive decomposition $ \mathfrak{g} = \mathfrak{h}\oplus \mathfrak{m} $, and denote by $ \hat{g} $ an $ \mathrm{Ad}(\mathsf{H}) $-invariant extension of $ g_{e\mathsf{H}} $ to $ \mathfrak{g} $ such that $ \mathfrak{h}\perp\mathfrak{m} $. Then, we say that $ \left( \mathsf{G}/\mathsf{H}, g \right) $ is \textit{standard} if the $ \hat{g} $-orthogonal complement of the nilradical of $ \mathfrak{g} $ is a Lie subalgebra of $ \mathfrak{g} $.
        \end{definition}

        We have the following result, relating the existence of cohomogeneity-one actions on $ (\mathsf{G}/\mathsf{H}, g) $ with the structure underlying $ \mathfrak{g} $.

        \begin{theorem}[\cite{bohm_lafuente_2019}*{Theorem D}]\label{thm:bohm_laf_2019_thm_d}
            Suppose $ (\mathsf{G}/\mathsf{H}, g) $ is a homogeneous Einstein manifold admitting an effective, cohomogeneity-one action of the closed subgroup $ \overline{\mathsf{G}} $ of $ \mathsf{G} $. If $ \overline{\mathsf{G}}\backslash \mathsf{G}/\mathsf{H} = \Sph^1 $ and consists of integrally minimal orbits, then the $ \overline{\mathsf{G}} $-orbits are standard homogeneous spaces.
        \end{theorem}

        Let $ \left( \mathsf{G}/\mathsf{H}, g \right) $ is as in the theorem above, with $ \overline{\mathsf{G}} $ an effective, cohomogeneity-one action on $ \mathsf{G}/\mathsf{H} $. Suppose $ \gamma: \Sph^1 \to \mathsf{G}/\mathsf{H} $ is a unit speed normal geodesic to the $ \overline{\mathsf{G}} $-orbits, and denote by $ \trace L_t $ the mean curvature of the hypersurfaces $ \Sigma_t = \overline{\mathsf{G}}\cdot\gamma(t) $. Then, we have the following
        \begin{definition}\label{def:integrally_minimal}
            In the notation above, we say that the orbits are \textit{integrally minimal} (in the sense of \cite{bohm_lafuente_2019}) if
            \[
                \int_{\Sph^1}\trace L_t \deint t = 0. 
            \]
        \end{definition}
        If $ \Sigma_t $ are all minimal hypersurfaces, then the orbits are integrally minimal. We say that $ \left( \mathsf{G}/\mathsf{H}, g, \overline{\mathsf{G}} \right) $ is \textit{orbit-Einstein} with negative Einstein constant $ \lambda < 0 $ if $ \mathrm{Ric}_g(X, X) = \lambda \cdot g(X, X) $ for all vectors $ X $ tangent to the orbits of $ \overline{\mathsf{G}} $. If $ (\mathsf{G}/\mathsf{H}, g) $ is Einstein with negative Einstein constant, then $ (\mathsf{G}/\mathsf{H}, g, \overline{\mathsf{G}}) $ is orbit-Einstein.

        Finally, let us recall the following formula for the Ricci curvature of a $ \mathsf{G} $-invariant metric $ g $, where $ \left\{ e_i \right\} $ is a $ g $-orthonormal basis for $ \mathfrak{m} $ \cite{besse_2008}*{(7.33)}:
        \begin{equation}\label{eq:ricci_formula}
            \begin{split}
                \mathrm{ric}(X, Y) =& -\tfrac12\sum_{ij}\iprod{\lb{X}{e_i}_\mathfrak{m}}{e_j}\iprod{\lb{Y}{e_i}_\mathfrak{m}}{e_j} + \tfrac14\sum_{ij}\iprod{\lb{e_i}{e_j}_\mathfrak{m}}{X}\iprod{\lb{e_i}{e_j}_\mathfrak{m}}{Y}\\
                &-\tfrac14\sum_i\iprod{\lb{X}{\lb{Y}{e_i}_\mathfrak{m}}_\mathfrak{m}}{e_i} - \tfrac14\sum_i\iprod{\lb{Y}{\lb{X}{e_i}_\mathfrak{m}}_\mathfrak{m}}{e_i}\\
                &-\tfrac12\sum_i\iprod{\lb{X}{\lb{Y}{e_i}_\mathfrak{h}}}{e_i}-\tfrac12\sum_i\iprod{\lb{Y}{\lb{X}{e_i}_\mathfrak{h}}}{e_i}\\
                &-\tfrac12\iprod{\lb{H}{X}_\mathfrak{m}}{Y} -\tfrac12\iprod{\lb{H}{Y}_\mathfrak{m}}{X} 
            \end{split}
        \end{equation}
        where $ \mathrm{H} $ is the mean curvature vector of $ (\mathsf{G}/\mathsf{H}, g) $. 

    }

    \section{Proof of Theorem \ref{thm:sl2r_theorem}}
    The purpose of this section is to prove Theorem \ref{thm:sl2r_theorem}. In particular, we prove that any Einstein metric on $ \mathsf{G}/\mathsf{H} $ is a Riemannian product of Einstein metrics. Thus the left invariant metric on the factor with Lie algebra isomorphic to $ \mathfrak{sl}_2(\R) $ is Einstein as well, which is impossible.
    
    Suppose $ (\mathcal{M}^n, g) $ is a simply-connected, non-compact homogeneous manifold presented as $ \mathsf{G}/\mathsf{H} $, with $ \mathsf{G} $ semisimple. Let $ \mathfrak{g}_0 $ be a simple ideal of $ \mathfrak{g} $ and $ \mathfrak{g}_1 $ the semisimple ideal complementary to $ \mathfrak{g}_0 $, and denote by $ \mathfrak{h}_0 \subseteq \mathfrak{g}_0 $ the projection of $ \mathfrak{h} $ onto the simple factor $ \mathfrak{g}_0 $ with respect to the splitting $ \mathfrak{g} = \mathfrak{g}_0 \oplus \mathfrak{g}_1 $. 
    \begin{lemma}\label{lem:trivial_projection_lemma}
        Under the above assumptions, if $ \mathfrak{h}_0 $ is a maximal compactly embedded subalgebra and $ \mathsf{Q}_0 $ a closed Lie subgroup of $ \mathsf{G} $ with $ \mathrm{Lie}(\mathsf{Q}_0) = \mathfrak{q}_0 \subseteq \mathfrak{g}_0 $ such that $ \mathfrak{g}_0 = \mathfrak{h}\oplus \mathfrak{q}_0 $ as vector spaces, then $ \mathsf{G}/\mathsf{H} $ is not a minimal presentation.
    \end{lemma}
    \begin{proof}
        Define $ \overline{\mathfrak{g}} = \mathfrak{q}_0\oplus \mathfrak{g}_1 $, and denote by $ \mathsf{Q}_0 $ and $ \overline{\mathsf{G}} $ the connected Lie subgroups of $ \mathsf{G} $ with Lie algebras $ \mathfrak{q}_0 $ and $ \overline{\mathfrak{g}} $ respectively. We claim that the action of $ \overline{\mathsf{G}} $ on $ \mathsf{G}/\mathsf{H} $ is transitive.

        It is sufficient to look at the infinitesimal generators of the action by \cite{audin_2004}, and show that the homomorphism $ \iota : \overline{\mathfrak{g}} \to \iota(\overline{\mathfrak{g}}) \subseteq \mathrm{Kill}(\mathsf{G}/\mathsf{H}) $ defined pointwise by 
        \[
            X \mapsto \iota(X)_p \coloneqq \diff{}{t}\Big\vert_{t = 0}\exp(tX)\cdot p,
        \]
        restricts to an epimorphism $ \iota(\cdot)_p : \overline{\mathfrak{g}} \to T_p\mathsf{G}/\mathsf{H} $ on every tangent space. Indeed, since $ T_p\mathsf{G}/\mathsf{H} \simeq \mathfrak{g}/\mathfrak{h} $, and $ \iota(\overline{\mathfrak{g}})_p \simeq \overline{\mathfrak{g}}/\ker\iota(\cdot)_p $ by the first isomorphism theorem, $ \iota(\cdot)_p $ is an epimorphism if and only if $ \dim\ker\iota(\cdot)_p = \dim\mathfrak{h} - (\dim\mathfrak{g} - \dim\overline{\mathfrak{g}}) $. Since

        \[
        	\overline{\mathsf{G}}_p \coloneqq \left\{ g \in \overline{\mathsf{G}} \;:\; g\cdot p = p \right\} = \left\{  g \in \mathsf{G} \;:\; g\cdot p = p\text{ and }g \in \overline{\mathsf{G}}  \right\} = \mathsf{H}\cap \overline{\mathsf{G}},
        \]

        by \cite{audin_2004}*{Theorem 1.1.1}, 
        \[
        	\ker \iota(\cdot)_p = \overline{\mathfrak{g}}_p = \mathrm{Lie}(\overline{\mathsf{G}}_p) = \overline{\mathfrak{g}}\cap \mathfrak{h}.
        \]
        Therefore, by the second isomorphism theorem, and the fact that $ \mathfrak{h} + \overline{\mathfrak{g}} = \mathfrak{g} $, we have
        \[
        	\dim \ker\iota(\cdot)_p = \dim(\overline{\mathfrak{g}}\cap \mathfrak{h}) = \dim\mathfrak{h} - (\dim\mathfrak{g} - \dim\overline{\mathfrak{g}}),
        \]
        and hence the action of $ \overline{\mathsf{G}} $ on $ \mathsf{G}/\mathsf{H} $ is transitive, so $ \mathsf{G}/\mathsf{H} $ is not a minimal presentation.
    \end{proof}

    Suppose now that $ \mathsf{G}/\mathsf{H} $ is a minimal presentation for $ \left( \mathcal{M}^n, g \right) $ in the above, with $ \mathfrak{g}_0 = \mathfrak{sl}_2(\R) $. Then, by Lemma \ref{lem:trivial_projection_lemma}, if $ \mathfrak{h}_0 $ is nontrivial, it is a maximal compactly embedded subalgebra. Then, any Borel subgroup $ \mathsf{B} = \mathsf{Q}_0 $ satisfies the conditions of the above Lemma, which is impossible. Hence, $ \mathfrak{h}_0 $ is trivial. 

    {
    If $ \mathcal{M} $ is simply connected, then we have the splitting $ \mathcal{M} = \widetilde{\mathsf{SL}_2(\R)}\times \mathsf{G}_1/\mathsf{H} $ as a homogeneous product. In order to apply \cite{bohm_lafuente_2019}*{Theorem D}, it is enough to prove a codimension 1 subgroup $ \overline{\mathsf{G}} $ of $ \mathsf{G} $ such that the cohomogeneity one action of $ \overline{\mathsf{G}} $ on $ \mathcal{M} $ has orbit space $ \Sph^1 $. If $ \mathfrak{b} $ is any Borel subalgebra of $ \mathfrak{sl}_2(\R) $ and $ \mathsf{B}\leq \widetilde{\mathsf{SL}_2(\R)} $ is the corresponding connected subgroup, then the orbit space of $ \overline{\mathsf{G}} = \mathsf{B} \times \mathsf{G}_1 $ in $ \mathcal{M} $ is non-compact. In order to apply \cite{bohm_lafuente_2019}*{Theorem D}, it is necessary to pass to the quotient of some subgroup of $ \mathsf{G} $ acting properly discontinuously on $ \mathcal{M} $.

    Suppose now that $ \mathcal{M} $ is not necessarily simply connected. Denote by $ \mathsf{G}_0 $ and $ \mathsf{G}_1 $ the connected Lie subgroups of $ \mathsf{G} $ corresponding to the Lie algebras $ \mathfrak{g}_0 = \mathfrak{sl}_2(\R) $ and $ \mathfrak{g}_1 $ respectively. Then, since $ \mathfrak{g} = \mathfrak{g}_0 \oplus \mathfrak{g}_1 $ is a decomposition into ideals, the intersection $ \mathsf{I} = \mathsf{G}_0 \cap \mathsf{G}_1 $ is a normal and discrete subgroup of $ \mathsf{G} $, and so $ \mathsf{I} $ is also central in $ \mathsf{G} $ \cite{onishchik_vinberg_1988}*{Proposition 4.6}. Therefore, $ \mathsf{\Lambda} = Z(\mathsf{G}_0)\mathsf{I} $ acts properly discontinuously on $ (\mathcal{M}, g) $ by isometries, so $ \mathsf{G}/\mathsf{H} $ is locally isometric to $ \mathsf{\Lambda}\backslash \mathsf{G}/\mathsf{H} \simeq \mathsf{PSL}_2(\R)\times (\mathsf{G}_1/\mathsf{H}') $, where $ \mathsf{H}' = \mathsf{H}\cap \mathsf{\Lambda} $. Since the Einstein condition is local, it suffices to prove Theorem \ref{thm:sl2r_theorem} in the case where $ \mathcal{M} $ is presented minimally as $ \mathsf{PSL}_2(\R)\times (\mathsf{G}_1/\mathsf{H}) $.

    }
    \begin{lemma}
        Let $ (\mathcal{M}^n, g) $ be a homogeneous Riemannian manifold with minimal presentation $ \mathsf{PSL}_2(\R) \times \mathsf{G}_1/\mathsf{H} $ as in the notation above. Let $ \mathfrak{q}_0 = \mathfrak{a}_0\oplus \mathfrak{n}_0\subseteq \mathfrak{sl}_2(\R) $ be a Borel subalgbera with matching Iwasawa decomposition $ \mathfrak{sl}_2(\R) = \mathfrak{k}_0\oplus \mathfrak{a}_0\oplus \mathfrak{n}_0 $, and denote $ \overline{\mathfrak{g}} = \mathfrak{q}_0 \oplus \mathfrak{g}_1 $ as in the above. Then
        \begin{enumerate}
            \item $ \overline{\mathsf{G}} $ is closed in $ \mathsf{G} $;
            \item the action of $ \overline{\mathsf{G}} $ on $ \mathcal{M} $ is effective, and of cohomogeneity-one;
            \item the orbits are integrally minimal;
            \item $ \overline{\mathsf{G}}\backslash\mathcal{M} = \Sph^1 $.
        \end{enumerate}
    \end{lemma}
    \begin{proof}
        Since $ \mathsf{Q}_0 $ and $ \mathsf{G}_1 $ are both closed, $ \overline{\mathsf{G}} = \mathsf{Q}_0\times \mathsf{G}_1 $ is closed. Moreover, the effective action of $ \mathsf{G} $ on $ \mathcal{M} $ restricts to an effective action of $ \overline{\mathsf{G}} $ on $ \mathcal{M} $.

        We claim that $ \overline{\mathsf{G}}\backslash\mathcal{M} = \Sph^1 $. Indeed, since the center of $ \mathsf{PSL}_2(\R) $ is trivial, $ \mathsf{K}_0 $ is compact by \cite{helgason_1978}*{Chapter 6, Theorem 1.1}, and so $ \mathsf{K}_0 \simeq \Sph^1 $. Writing $ \mathsf{G} = (\mathsf{K}_0 \mathsf{Q}_0) \times \mathsf{G}_1 $, we have that every element in $ \mathcal{M} $ is written uniquely as $ k\cdot(q, g_1)\mathsf{H} = (kq, g_1)\mathsf{H} $ for some $ k \in \mathsf{K}_0 $, $ q \in \mathsf{Q}_0 $, and $ g_1 \in \mathsf{G}_1 $. That is, $ \mathcal{M} \simeq \mathsf{K}_0(\overline{\mathsf{G}}/\mathsf{H}) $, so the orbits of the $ \overline{\mathsf{G}} $ action on $ \mathcal{M} $ are parametrised by the action of $ \mathsf{K}_0 $ on $ \overline{\mathsf{G}}/\mathsf{H} $, so $ \mathcal{M}\backslash \overline{\mathsf{G}} = \mathsf{K}_0 = \Sph^1 $.

        It remains to show integral minimality in the sense of \cite{bohm_lafuente_2019}. We will show more generally that the $ \overline{\mathsf{G}} $-orbits are minimal hypersurfaces. Let $ \gamma : \Sph^1 \to \mathcal{M} $ be a unit speed normal geodesic to the orbits of the $ \overline{\mathsf{G}} $-action, and set $ \Sigma_t = \overline{\mathsf{G}}\cdot \gamma(t) $. Let $ N $ denote the unit normal vector field to each $ \Sigma_t $, satisfying $ \gamma'(t) = N_{\gamma(t)} $ for all $ t \in \Sph^1 $, and $ X $ denote a vector field tangent to the $ \overline{\mathsf{G}} $-orbits. Now, locally we have that $ g(X, N) = 0 $, and so $ g(\nabla_X N,X) = -g(N, \nabla_X X) $.
        
        Following the same construction as \cite{bohm_lafuente_2019}*{\S1}, the shape operator of $ \Sigma_t $, denoted $ L_t \in \mathrm{End}(T\Sigma_t) $ satisfies
        \[
        	g_t(L_tX, Y) = -g(\nabla_X N, Y)_{\gamma(t)} = g(N, \nabla_X Y)_{\gamma(t)}
        \]
        for Killing fields $ X, Y \in \overline{\mathfrak{g}} $, where $ g_t $ is the induced metric on the submanifold $ \Sigma_t $ of $ \mathcal{M} $. Replacing $ N $ with a Killing field $ N^* \in \mathfrak{g} $ with the same value as $ N $ at $ \gamma(t) $, the Koszul formula for Killing fields \cite{besse_2008}*{(7.27)} implies that for all Killing fields $ X \in \mathfrak{g} $,
        \begin{align*}
            g_t(L_tX, X) & = g(N^*, \nabla_X X)_{\gamma(t)}\\
            & = \tfrac12\left( g(N^*, \lb{X}{X})_{\gamma(t)} + g(X, \lb{X}{N^*})_{\gamma(t)} + g(X, \lb{X}{N^*})_{\gamma(t)} \right)\\
            & = -g(\lb{N^*}{X}, X)_{\gamma(t)}
        \end{align*}
        Let $ \mathfrak{g} = \mathfrak{h}\oplus \mathfrak{m} $ and $ \mathfrak{g}_1 = \mathfrak{h}\oplus \mathfrak{m}_1 $ be the canonical reductive decompositions for $ \mathsf{G}/\mathsf{H} $ and $ \mathsf{G}_1/\mathsf{H} $, and denote by $ \overline{\mathfrak{m}} = \mathfrak{q}_0\oplus \mathfrak{m}_1 = T_{\gamma(t)}\Sigma_t $. By adding elements from the isotropy, we may assume that $ N^* \in \mathfrak{m} $. Suppose $ \{U_i\}_{i = 1}^{n-1} $ is an orthonormal basis for $ \overline{\mathfrak{m}} $. Then $ \left\{ U_i \right\}\cup \left\{ N^* \right\}  $ is an orthonormal basis for $ \mathfrak{m} $. If $ \left\{ V_i \right\} $ is an orthonormal basis for $ \mathfrak{h} $, then
        \begin{align*}
            0 = -\trace\ad_{N^*} &= -\sum_ig(\lb{N^*}{U_i}, U_i)_{\gamma(t)} - \sum_j g(\lb{N^*}{V_j}, V_j)_{\gamma(t)}\\
            & = \sum_ig(\lb{N^*}{U_i},U_i)_{\gamma(t)} = \trace L_t
        \end{align*}
        since $ \lb{\mathfrak{m}}{\mathfrak{h}} \subseteq \mathfrak{m} $, and the fact that $ \mathfrak{sl}_2(\R) $ is unimodular. Hence, $ \Sigma_t $ is a minimal hypersurface.
    \end{proof}
    


    { 
        Suppose now that $ (\mathcal{M}, g) $ is in addition Einstein with negative Einstein constant. Then, by Theorem \ref{thm:bohm_laf_2019_thm_d}, the cohomogeneity-one $ \overline{\mathsf{G}} $-orbits are standard homogeneous spaces. 
    }

    \begin{lemma}\label{lem:borel_orthogonality}
        In the notation above and of Definition \ref{def:standard}, if the $ \overline{\mathsf{G}} $-orbits are standard, then $ \mathfrak{g}_1 $ is $ \hat{g} $-orthogonal to the nilradical of $ \mathfrak{q}_0 $.
    \end{lemma}
    \begin{proof}
        It is sufficient to prove that $ \mathfrak{g}_1 \subseteq \mathfrak{u} \coloneqq \mathfrak{n}_0^\perp $, the orthogonal complement of $ \mathrm{nilrad}(\mathfrak{g}) = \mathfrak{n}_0 $ under the $ \mathrm{Ad}(\mathsf{H}) $-invariant extension of the inner product $ g $ to $ \mathfrak{g} $. Since the $ \overline{\mathsf{G}} $-orbits are standard, $ \overline{\mathfrak{g}} = \mathfrak{u} \oplus \mathfrak{n}_0 $ is a Lie algebra direct sum. Moreover, $ \mathfrak{u} \simeq \overline{\mathfrak{g}}/\mathfrak{n}_0 $ is reductive by \cite{varadarajan_1984}*{Theorem 3.16.3}, and so $ \mathfrak{u}^{(1)} =\lb{\mathfrak{u}}{\mathfrak{u}} $ is semisimple, and
        \begin{equation}\label{eq:g_commutation_relations}
            \overline{\mathfrak{g}}^{(1)} = \lb{\overline{\mathfrak{g}}}{\overline{\mathfrak{g}}}\subseteq \mathfrak{n}_0\oplus \lb{\mathfrak{u}}{\mathfrak{u}} = \mathfrak{n}_0\oplus \mathfrak{u}^{(1)}.
        \end{equation}
        
        Now, the bracket $ \lb{\mathfrak{u}^{(1)}}{\mathfrak{n}_0} $ defines $ \mathfrak{n}_0 $ as a trivial 1-dimensional $ \mathfrak{u}^{(1)} $-representation, since $ \mathfrak{u}^{(1)} $ is semisimple. Hence, $ \mathfrak{n}_0 \oplus \mathfrak{u}^{(1)} $ is a Lie algebra direct sum. Also, since $ \overline{\mathfrak{g}} = \mathfrak{q}_0 \oplus \mathfrak{g}_1 $ is a decomposition into ideals and $ \mathfrak{q}_0 $ is at most 2-step solvable,
        \[
        	\mathfrak{g}_1 = \overline{\mathfrak{g}}^{(2)} \overset{\eqref{eq:g_commutation_relations}}{\subseteq} \lb{\mathfrak{n}_0\oplus \mathfrak{u}^{(1)}}{\mathfrak{n}_0\oplus \mathfrak{u}^{(1)}} \subseteq \mathfrak{u}^{(2)} \subseteq \mathfrak{u},  
        \]
        as required.
    \end{proof}

    \begin{proof}[Proof of Theorem \ref{thm:sl2r_theorem}]
        As in the above notation, let $ \left( \mathcal{M}, g \right) $ be a simply-connected homogeneous Einstein manifold of negative Ricci curvature, with minimal presentation $ \mathsf{PSL}_2(\R)\times \left( \mathsf{G}_1/\mathsf{H} \right) $ such that $ \mathsf{G}_1 $ is semisimple. We seek three Borel subalgebras of $ \mathfrak{g}_0 = \mathfrak{sl}_2(\R) $ whose nilradicals form a basis. In the basis $ \left\{ \hat h, \hat e, \hat f \right\} $ of $ \mathfrak{sl}_2(\R) $, with canonical commutation relations
        \[
        	\lb ef = h, \quad\lb he = 2e, \quad\lb hf = -2f,
        \]
        three such Borel subalgebras take the form
        \[
        	\mathfrak{q}_0^1 = \Span\left\{ h, e \right\} ,\quad \mathfrak{q}_0^2 = \Span\left\{ h, f \right\},\quad \mathfrak{q}_0^3 = \Span\left\{ e + f, e - f + h \right\}.
        \]
        Applying Lemma \ref{lem:borel_orthogonality} on each $ \mathfrak{q}_0^i $, we see that $ \mathfrak{sl}_2(\R) \perp \mathfrak{g}_1 $. Therefore, $ \mathsf{G}/ \mathsf{H} = \mathsf{PSL}_2(\R) \times \mathsf{G}_1 / \mathsf{H} $ splits as a Riemannian product, and hence the Einstein equations splits into the Einstein equation on the $ \mathsf{PSL}_2(\R) $ and $ \mathsf{G}_1/\mathsf{H} $ factors. But it is well known (for example \cite{milnor_1976}) that $ \mathsf{PSL}_2(\R) $ admits no left invariant Einstein metrics of negative Einstein constant. 
    \end{proof}

    \section{Classification of unresolved spaces}\label{sec:classification}
    The main purpose of this section is to provide a classification of the remaining cases necessary to prove Theorem \ref{thm:main_alek_theorem}. Recall that $ \mathcal{M}^n = \mathsf{G}/\mathsf{H} $ is a simply connected, (almost) effective minimal presentation of a homogeneous space, with $ \mathsf{G} $ and $ \mathsf{H} $ connected and $ \mathsf{H} $ compact. Assume in addition that $ \left( \mathcal{M}, g \right) $ is an Einstein manifold of negative Ricci curvature, of dimension at most 10. Then, by \citelist{\cite{arroyo_lafuente_2016}\cite{bohm_lafuente_2019}\cite{jablonski_petersen_2017}\cite{jensen_1969}\cite{lafuente_lauret_2014}\cite{nikonorov_2005}} and references therein, we may assume $ \dim \mathcal{M} $ is 9 or 10. Moreover, we may assume without loss of generality that for the minimal presentation $ \mathsf{G}/\mathsf{H} $ for $ \mathcal{M} $, $ \mathsf{G} $ is semisimple, and has no compact simple factors. By Theorem \ref{thm:sl2r_theorem}, $ \mathfrak{g} $ has no ideal isomorphic to $ \mathfrak{sl}_2(\R) $. We may also assume $ \mathsf{G}/\mathsf{H} $ is not an irreducible symmetric spaces, since these are all diffeomorphic to simply-connected solvmanifolds by \cite{helgason_1978}*{Chapter 6, Theorem 5.1}. 

    \begin{definition}
        Let $ \mathsf{G}/\mathsf{H} $ be a homogeneous space in minimal presentation. If $ \mathfrak{g} = \mathrm{Lie}(\mathsf{G}) $ is semisimple, we call $ \mathsf{G}/\mathsf{H} $ a semisimple homogeneous space.
    \end{definition}

    In order to obtain our classification, we use the partial duality between compact semisimple homogeneous spaces and semisimple homogeneous spaces without compact simple factors, in the sense of \citelist{\cite{alekseevsky_2012}\cite{nikonorov_2005}} to construct a list of all semisimple homogeneous spaces without compact simple factors. We will detail this procedure below:

    Let $ \widetilde{\mathsf{G}}/\widetilde{\mathsf{H}} $ be a compact semisimple homogeneous space, with $ \widetilde{\mathfrak{g}} = \mathrm{Lie}(\widetilde{\mathsf{G}}) = \widetilde{\mathfrak{g}} = \widetilde{\mathfrak{g}}_1\oplus \cdots \oplus \widetilde{\mathfrak{g}}_r $ for $ \widetilde{\mathfrak{g}}_i $ simple, and denote by $ \widetilde{\mathfrak{h}}_i $ the projection of the isotropy subalgebra onto $ \widetilde{\mathfrak{g}}_i $ in the decomposition. Suppose $ (\widetilde{\mathfrak{g}}, \mathfrak{k}) $ is a symmetric pair with $ \mathfrak{k} $ containing $ \widetilde{\mathfrak{h}} $. Denote $ \mathfrak{k} = \mathfrak{k}_1\oplus \cdots \oplus \mathfrak{k}_r $ with each $ \mathfrak{k}_i $ the projection of $ \mathfrak{k} $ onto $ \widetilde{\mathfrak{g}}_i $ as before. Then, each $ (\widetilde{\mathfrak{g}_i}, \mathfrak{k}_i) $ is an irreducible symmetric space of the compact type. For each $ i $, denote by $ \mathfrak{g}_i $ the corresponding non-compact symmetric dual of $ \widetilde{\mathfrak{g}_i} $, and $ \mathfrak{h}_i $ the corresponding subalgebra isomorphic to $ \widetilde{\mathfrak{h}}_i $ contained in $ \mathfrak{g}_i $. Then, the Lie algebra $ \mathfrak{g} = \mathfrak{g}_1 \oplus \cdots \oplus \mathfrak{g}_r $ gives rise to a simply-connected Lie group $ \mathsf{G} $ and a connected subgroup $ \mathsf{H} $ such that $ \mathfrak{h} = \mathrm{Lie}(\mathsf{H}) $. The corresponding homogeneous space $ \mathsf{G}/\mathsf{H} $ has no compact simple factors. Therefore, the classification problem reduces to finding all possible subalgebras of $ \mathfrak{k} $ for each symmetric pair of the compact type, $ (\widetilde{\mathfrak{g}}, \mathfrak{k}) $ with $ \dim\mathsf{G}/\mathsf{K} \leq 10 $ \cite{nikonorov_2005}.
    
    \begin{definition}\label{def:auspace}
        We call a minimally presented homogeneous space $ \mathsf{G}/\mathsf{H} $ \textit{Alekseevskii Unresolved} (AU) if it is simply-connected, non-compact, non-symmetric, non-product, with $ \mathsf{G} $ semisimple without compact simple factors, and without $ \mathfrak{sl}_2(\R) $ ideals in $ \mathfrak{g} $. Analogously, if $ \mathsf{G}/\mathsf{H} $ is simply-connected, compact, non-symmetric, non-product, semisimple without $ \mathfrak{su}(2) $ ideals in $ \mathfrak{g} $, then we say that $ \mathsf{G}/\mathsf{H} $ is AU*.
    \end{definition}

    Using \cite{besse_2008}, we have the following list of symmetric spaces of the compact type in dimensions at most 10. We ignore those spaces with a factor of $ \mathsf{SU}(2)/\mathsf{U}(1) $, since the non-compact duals of any homogeneous space generated from these are not Einstein by Theorem \ref{thm:sl2r_theorem}. Furthermore, we ignore the Lie group cases, since our methods are not equipped to deal with their non-compact duals.

    {\scriptsize
	\begin{table}[h]
        \begin{tabular}{*3c}\toprule
			$ \dim \quotient{\mathsf{G}}{\mathsf{K}} $ & Symmetric Space & $ \dim\mathsf{G} $\\\midrule

            \multirow{1}{*}{3}
            & $ \quotient{\mathsf{SU(2)\times\mathsf{SU}(2)}}{\Delta\mathsf{SU}(2)} $&6\\\midrule
			
            \multirow{2}{*}{4}
			& $ \quotient{\mathsf{SU}(3)}{\mathsf{U}(2)} $&8\\
			& $ \quotient{\mathsf{Sp}(2)}{\mathsf{Sp}(1)^{2}} $&10\\\midrule
			
			\multirow{2}{*}{5}
			& $ \quotient{\mathsf{SU}(3)}{\mathsf{SO}(3)} $&8\\
			& $ \quotient{\mathsf{SU}(4)}{\mathsf{Sp}(2)} $&15\\\midrule

            \multirow{4}{*}{6}
			& $ \quotient{\mathsf{SU}(4)}{\mathsf{U}(3)} $&15\\
			& $ \mathbb{S}^{6}\simeq \quotient{\mathsf{SO}(7)}{\mathsf{SO}(6)} $ & 21\\
			& $ \quotient{\mathsf{Sp}(2)}{\mathsf{U}(2)} $ & 10\\
			& $ \left(\quotient{\mathsf{SU}(2)\times \mathsf{SU}(2)}{\Delta\mathsf{SU}(2)}\right)^{2} $&12\\\midrule
			
			\multirow{3}{*}{7}
			& $ (\quotient{\mathsf{SU}(3)}{\mathsf{U}(2)})\times(\quotient{\mathsf{SU}(2)\times\mathsf{SU}(2)}{\Delta\mathsf{SU}(2)}) $&14\\
			& $ (\quotient{\mathsf{Sp}(2)}{\mathsf{Sp}(1)^{2}})\times (\quotient{\mathsf{SU}(2)\times \mathsf{SU}(2)}{\Delta\mathsf{SU}(2)}) $&16\\
            & $ \mathbb{S}^{7} \simeq \quotient{\mathsf{SO}(8)}{\mathsf{SO}(7)} $&28\\\midrule
			
			\multirow{12}{*}{8}
			& $ (\quotient{\mathsf{SU}(3)\times \mathsf{SU}(3)}{\Delta\mathsf{SU}(3)}) $&16\\
			& $ (\quotient{\mathsf{SU}(3)}{\mathsf{SO}(3)})\times (\quotient{\mathsf{SU(2)\times \mathsf{SU}(2)}}{\Delta\mathsf{SU}(2)}) $&14 \\
			& $ (\quotient{\mathsf{SU}(4)}{\mathsf{Sp}(2)})\times (\quotient{\mathsf{SU}(2)\times \mathsf{SU}(2)}{\Delta\mathsf{SU}(2)}) $&21\\
			& $ (\quotient{\mathsf{SU}(3)}{\mathsf{U}(2)})^{2} $&16\\
			& $ (\quotient{\mathsf{SU}(3)}{\mathsf{U}(2)})\times (\quotient{\mathsf{Sp}(2)}{\mathsf{Sp}(1)^{2}}) $&18\\
			& $ (\quotient{\mathsf{Sp}(2)}{\mathsf{Sp}(1)^{2}})^{2} $&20\\
			& $ \quotient{\mathsf{SU}(3)\times \mathsf{SU}(3)}{\Delta\mathsf{SU}(3)} $&16\\
			& $ \quotient{\mathsf{SU}(4)}{\mathsf{S}(\mathsf{U}(2)^{2})} $&15\\
			& $ \quotient{\mathsf{SU}(5)}{\mathsf{U}(4)} $&24\\
			& $ \quotient{\mathsf{Sp}(3)}{\mathsf{Sp}(2)\mathsf{Sp}(1)} $&21\\
			& $ \quotient{\mathsf{G}_{2}}{\mathsf{SU}(2)^{2}} $&14\\
			& $ \mathbb{S}^{8} \simeq\quotient{\mathsf{SO}(9)}{\mathsf{SO}(8)} $&36\\\midrule
			
			\multirow{10}{*}{9}
			&$ (\quotient{\mathsf{Sp}(2)}{\mathsf{U}(2)}) \times (\quotient{\mathsf{SU}(2)\times \mathsf{SU}(2)}{\Delta\mathsf{SU}(2)}) $&16\\
			&$ (\quotient{\mathsf{SO}(7)}{\mathsf{SO}(6)}) \times (\quotient{\mathsf{SU}(2)\times \mathsf{SU}(2)}{\Delta\mathsf{SU}(2)}) $&27\\
			&$ (\quotient{\mathsf{SU}(4)}{\mathsf{U}(3)}) \times (\quotient{\mathsf{SU}(2)\times \mathsf{SU}(2)}{\Delta\mathsf{SU}(2)}) $&21\\
			&$ (\quotient{\mathsf{SU}(3)}{\mathsf{SO}(3)})\times(\quotient{\mathsf{SU}(3)}{\mathsf{U}(2)}) $&16\\
			&$ (\quotient{\mathsf{SU}(3)}{\mathsf{SO}(3)})\times(\quotient{\mathsf{Sp}(2)}{\mathsf{Sp}(1)^{2}}) $&18\\
			&$ (\quotient{\mathsf{SU}(4)}{\mathsf{Sp}(2)}) \times (\quotient{\mathsf{SU}(3)}{\mathsf{U}(2)})$&23\\
			&$ (\quotient{\mathsf{SU}(4)}{\mathsf{Sp}(2)}) \times(\quotient{\mathsf{Sp}(2)}{\mathsf{Sp}(1)^{2}})$&25\\
			&$ (\quotient{\mathsf{SU}(2)\times \mathsf{SU}(2)}{\Delta\mathsf{SU(2)}})^{3} $&18\\
			& $ \quotient{\mathsf{SU}(4)}{\mathsf{SO}(4)} $&15\\
			& $ \mathbb{S}^{9} \simeq\quotient{\mathsf{SO}(10)}{\mathsf{SO}(9)} $&45\\\bottomrule
		\end{tabular}
	\caption{\label{tab:symm_space}Symmetric spaces of the compact type}
	\end{table}
}
    By considering subalgebras $ \mathfrak{h} \subseteq \mathfrak{k} $, classifications attained by \cite{bohm_kerr_2004}, and the embeddings of the corresponding $ \mathsf{H} $ in $ \mathsf{K} \leq \mathsf{G} $, we may attain a complete list of Alekseevskii unresolved spaces in dimensions 9 and 10 in the sequel.

    In the following classification, we do not consider minimality of the presentation. 
    \begin{proposition}\label{cor:no_sl2c_minimals}
        Suppose $ \mathsf{G}/\mathsf{H} $ is a semisimple homogeneous space with an ideal $ \mathfrak{g}_0 $ isomorphic to $ \mathfrak{sl}_2(\C) $ of $ \mathfrak{g} $. Then, if the projection of $ \mathfrak{h} $ onto $ \mathfrak{g}_0 $ is isomorphic to $ \mathfrak{su}(2) $, there exists a closed Lie subgroup $ \overline{\mathsf{G}} $ of smaller dimension than $ \mathsf{G} $ acting transitively on $ \mathcal{M} $.
    \end{proposition}
    \begin{proof}
        Since the projection $ \mathfrak{h}_0 $ of $ \mathfrak{h} $ onto $ \mathfrak{sl}_2(\C) $ is a maximal compactly embedded subalgebra of $ \mathfrak{sl}_2(\C) $, there exists a Borel subgroup $ \mathsf{B} = \mathsf{Q}_0 $ such that $ \mathfrak{sl}_2(\C) = \mathfrak{h}_0 \oplus \mathfrak{b} $. Therefore, $ \mathsf{G}/\mathsf{H} $ is not a minimal presentation by Lemma \ref{lem:trivial_projection_lemma}.
    \end{proof}

    Moreover, since $ \dim\mathsf{SU}(3) = 8 $, there are no 9 or 10 dimensional homogeneous spaces with an $ \mathsf{SU}(3) $ transitive group. We disregard these cases.
    
    \subsection{$ \dim\quotient{\mathsf{G}}{\mathsf{K}} = 3 $}
    For the symmetric space $ \mathsf{SU}(2)\times \mathsf{SU}(2)/\Delta\mathsf{SU}(2) $, no examples exist, since $ \dim \mathsf{G} = 6 < 9, 10 $.

    \subsection{$ \dim\mathsf{G}/\mathsf{K} = 4 $} We have the following examples of 9 and 10 dimensional compact homogeneous spaces.

    \subsubsection{$ \quotient{\mathsf{G}}{\mathsf{K}} = \quotient{\mathsf{Sp}(2)}{\mathsf{Sp}(1)^{2}} $}
    Since $ \mathsf{Sp}(2) $ is 10 dimensional, we need a 1 dimensional subalgebra of $ \mathfrak{sp}(1)\oplus \mathfrak{sp}(1) $. Indeed, for every coprime $ p, q \in \Z $, we have the compact 9 dimensional homogeneous space $ \mathsf{Sp}(2)/\Delta_{p, q}\mathsf{U}(1) $. It has non-compact dual $ \mathsf{Sp}(1, 1)/\Delta_{p, q}\mathsf{U}(1) $, and the isotropy representation is given by $ \mathfrak{q}\oplus \mathfrak{p} = \qrep01\oplus\qrep12\oplus\qrep22\oplus\prep12\oplus\prep22 $, where
    \begin{multicols}{2}
        \noindent
        \begin{align*}
            \qrep01 &= \left\{ \begin{bmatrix}
                -qxi & 0\\
                0 & qxi
            \end{bmatrix},\, x\in \R \right\},\\
            \qrep12 &= \left\{ \begin{bmatrix}
                zj & 0\\
                0 & 0
            \end{bmatrix},\, z \in \C \right\},\\
            \qrep22 &= \left\{ \begin{bmatrix}
                0 & 0\\
                0 & zj
            \end{bmatrix},\, z\in\C \right\},
        \end{align*}
        \begin{align*}
            \prep12 &= \left\{ \begin{bmatrix}
                0 & z\\
                \overline{z} & 0
            \end{bmatrix},\, z \in \C \right\},\\
            \prep22 &= \left\{ \begin{bmatrix}
                0 & zj\\
                -zj & 0
            \end{bmatrix},\, z \in \C \right\},
        \end{align*}
        \vfill
    \end{multicols}
    and $ \qrep12 \simeq \qrep22 \simeq \prep22 \iff p = q = 1 $, and $ \prep22\simeq \prep12 \iff p = 0, q = 1 $.

    \subsection{$ \dim\mathsf{G}/\mathsf{K} = 5 $} We have the following examples.
    
    \subsubsection{$ \quotient{\mathsf{G}}{\mathsf{K}} = \quotient{\mathsf{SU}(4)}{\mathsf{Sp}(2)} $}
    The subalgebra $ \mathfrak{h} $ of $ \mathfrak{sp}(2) $ with $ \dim \mathfrak{h} \geq 5 $ is $ \mathfrak{sp}(1)\oplus \mathfrak{sp}(1) $, which gives rise to the homogeneous space $ \quotient{\mathsf{SU}(4)}{\mathsf{Sp}(1)\mathsf{Sp}(1)} $. It has non-compact dual $ \mathsf{SL}_2(\Quaternion)/\mathsf{Sp}(1)\mathsf{Sp}(1) $, with isotropy representation decomposing as $ \mathfrak{q}\oplus\mathfrak{p} = \qrep14 \oplus \prep01\oplus\prep14 $, where
    \[
        \qrep14 = \left\{ \begin{bmatrix}
            0 & h_1\\
            -\overline{h_1} & 0
        \end{bmatrix}\;:\; h_1 \in \Quaternion \right\},\quad
        \prep01 = \left\{ \begin{bmatrix}
            z & 0 \\
            0 & -z
        \end{bmatrix}\;:\; z \in \R \right\},
    \]
    \[
        \prep14 = \left\{ \begin{bmatrix}
            0 & \overline{h_2}\\
            h_2 & 0
        \end{bmatrix}\;:\; h_2 \in \Quaternion \right\},
    \]
    and $ \qrep14 \simeq \prep14 $.

    \subsection{$ \dim\mathsf{G}/\mathsf{K} = 6 $} We have the following examples

    \subsubsection{$ (\mathsf{SU}(2)\times\mathsf{SU}(2)/\Delta\mathsf{SU}(2))^2 $}
    We have the family of 10 dimensional homogeneous spaces $ \quotient{\mathsf{SU}(2)^{4}}{\Delta_{p, q}\mathsf{U}(1)\times \Delta_{r, s}\mathsf{U}(1)} $ and the 9 dimensional homogeneous space $ \quotient{\mathsf{SU}(2)^{4}}{\Delta\mathsf{SU}(2)} $, where $ \Delta\mathsf{SU}(2) $ lives diagonally in $ \mathsf{SU}(2)\times \mathsf{SU}(2) $.

    The non-compact dual of $ \quotient{\mathsf{SU}(2)^{4}}{\Delta_{p, q}\mathsf{U}(1)\times \Delta_{r, s}\mathsf{U}(1)} $ is $ \mathsf{SL}_2(\C)\times\mathsf{SL}_2(\C)/\mathsf{U}(1)^2 $, with isotropy representation $ \mathfrak{q}\oplus\mathfrak{p} = \qrep12\oplus\qrep22\oplus\prep02\oplus\prep12\oplus\prep22 $, with $ \qrep12\simeq\prep12 $ and $ \qrep22 \simeq \prep22 $, where
    \begin{align*}
        {\prep02} = \mathrm{span}\left\{ 
            D_1, D_2
        \right\},&\quad
        {\qrep12} = \mathrm{span}\left\{ 
            A_1, iS_1
        \right\},\\
        {\qrep22} = \mathrm{span}\left\{ 
            A_2, iS_2
        \right\},&\quad
        {\prep12} = \mathrm{span}\left\{ 
            S_1, iA_1
        \right\},\\
        {\prep22} &= \mathrm{span}\left\{ 
            S_2, iA_2
        \right\}
    \end{align*}
    such that
    \[
    	D_i = \begin{bmatrix}
            1 & 0 \\ 0 & -1
        \end{bmatrix},\quad
        A_i = \begin{bmatrix}
            0 & 1\\
            -1 & 0
        \end{bmatrix},\quad
        S_i = \begin{bmatrix}
            0 & 1\\
            1 & 0
        \end{bmatrix}
    \]
    are elements of the $ j $\textsuperscript{th} copy of $ \mathfrak{sl}_2(\C) $.

    The non-compact dual of $ \quotient{\mathsf{SU}(2)^{4}}{\Delta\mathsf{SU}(2)} $ is $ \mathsf{SL}_2(\C)\times\mathsf{SL}_2(\C)/\Delta\mathsf{SU}(2) $, but this space is not minimally presented by Proposition \ref{cor:no_sl2c_minimals}.

    \subsubsection{$ \quotient{\mathsf{G}}{\mathsf{K}} = \quotient{\mathsf{SU}(4)}{\mathsf{U}(3)} $}
    The subalgebras $ \mathfrak{h} \subseteq \mathfrak{u}(3) \subseteq \mathfrak{su}(4) $ with $ \dim \mathfrak{h} \geq 5 $ are $ \mathfrak{su}(3) $, and $ \mathfrak{s}(\mathfrak{u}(1)\oplus \mathfrak{u}(1)\oplus \mathfrak{u}(2)) $, which give rise to the homogeneous spaces $ \quotient{\mathsf{SU}(4)}{\mathsf{SU}(3)} $ and $ \quotient{\mathsf{SU}(4)}{\mathsf{S}(\mathsf{U}(1)\mathsf{U}(1)\mathsf{U}(2))} $. The former is 7 dimensional, and so we may ignore it. The latter has non-compact dual $ \mathsf{SU}(3, 1)/\mathsf{S}(\mathsf{U}(1)\mathsf{U}(1)\mathsf{U}(2)) $, and isotropy representation $ \mathfrak{q}\oplus \mathfrak{p} = \qrep14\oplus\prep12\oplus\prep24 $, where
    \begin{align*}
        \qrep24 &= \left\{
        \begin{bsmallmatrix}
        -i(a + b + c) & 0 & 0 & 0\\
        0 & ia & 0 & 0\\
        0 & 0 & ib & z\\
        0 & 0 & -\overline{z} & ic
        \end{bsmallmatrix},\, a, b, c \in \R, z \in \C\right\}\\
        \prep12 &= \left\{
        \begin{bsmallmatrix}
        0 & 0 & 0 & 0\\
        0 & 0 & y_1 & y_2\\
        0 & -\overline{y_1} & 0 & 0\\
        0 & -\overline{y_{2}} & 0 & 0
        \end{bsmallmatrix},\, y_i \in \R
        \right\},\quad\prep24 =
        \left\{\begin{bsmallmatrix}
        0 & x_1 & x_2 & x_3\\
        -\overline{x_1} & 0 & 0 & 0\\
        -\overline{x_2} & 0 & 0 & 0\\
        -\overline{x_3} & 0 & 0 & 0
        \end{bsmallmatrix},\, x_i \in \R\right\}.
    \end{align*}

    \subsubsection{$ \quotient{\mathsf{G}}{\mathsf{K}} = \quotient{\mathsf{SO}(7)}{\mathsf{SO}(6)} $}
    We are looking for subgroups $ \mathsf{H} $ of $ \mathsf{SO}(6) $ with $ \dim \mathsf{H} \geq 11 $. Since $ \dim\mathsf{SO}(6) = 15 $, we need a subalgebra of $ \mathfrak{so}(6) $ of dimension at least 11. But, since $ \mathfrak{so}(6) = \mathfrak{su}(4) $, and the maximal subalgebras of $ \mathfrak{su}(4) $ are $ \mathfrak{u}(3) = \mathfrak{s}(\mathfrak{u}(1)\oplus \mathfrak{u}(3)) $, $ \mathfrak{s}(\mathfrak{u}(2)\oplus \mathfrak{u}(2)) $, $ \mathfrak{sp}(2) $ and $ \mathfrak{so}(4) $, which are all of dimension at most 10. Therefore, there are no subgroups of $ \mathsf{SO}(6) $ with dimension at least 11, so there are no homogeneous spaces of the form we are looking for with associated symmetric space $ \quotient{\mathsf{SO}(7)}{\mathsf{SO}(6)} $.

    \subsubsection{$ \quotient{\mathsf{G}}{\mathsf{K}} = \quotient{\mathsf{Sp}(2)}{\mathsf{U}(2)} $}
    We need a 1 dimensional subalgebra of $ \mathfrak{u}(2) $.
    The only example of this is $ \Delta_{p, q}\mathfrak{u}(1) $, for $ p, q \in \Z $ coprime. This gives rise to the homogeneous space $ \quotient{\mathsf{Sp}(2)}{\Delta_{p, q}\mathsf{U}(1)} $, which has non-compact dual $ \mathsf{Sp}(2, \R)/\Delta_{p, q}\mathsf{U}(1) $. The isotropy representation is $ \mathfrak{q}\oplus \mathfrak{p} = \qrep01\oplus\qrep12\oplus\prep12\oplus\prep22\oplus\prep23 $, with
    \begin{align*}
        \qrep12 = \left\{\begin{bmatrix}
            0 & z\\
            -\overline{z} & 0
        \end{bmatrix},\, z \in \C\right\},&\quad
        \prep12 = \left\{\begin{bmatrix}
            zj & 0\\
            0 & 0
        \end{bmatrix},\, z \in \C\right\},\\
        \prep22 = \left\{\begin{bmatrix}
            0 & zj\\
            zj & 0
        \end{bmatrix},\, z \in \C\right\},&\quad
        \prep32 = \left\{\begin{bmatrix}
            0 & 0\\
            0 & zj
        \end{bmatrix},\, z \in \C\right\},
    \end{align*}
    where $ \prep12 \simeq \prep22 \simeq \prep32 $ if and only if $ p = q = 1 $, and $ \prep22 \simeq \prep32 $ if and only if $ p = 0 $ and $ q = 1 $. Hence, this space admits no invariant Einstein metrics by Theorem \ref{thm:nikonorov_thm_1}

    \subsection{$ \dim\mathsf{G}/\mathsf{K} = 7 $} We have several examples in this case, which we elaborate on below.

    \subsubsection{$ \mathsf{G}/\mathsf{K} = (\quotient{\mathsf{SU}(3)}{\mathsf{U}(2)})\times(\quotient{\mathsf{SU}(2)\times\mathsf{SU}(2)}{\Delta\mathsf{SU}(2)}) $}
    We can embed a one dimensional torus diagonally in $ \mathsf{U}(2)\times \mathsf{SU}(2) $, and attain the compact homogeneous space
    \[ \quotient{\mathsf{SU}(3)\times \mathsf{SU}(2)\times \mathsf{SU}(2)}{\Delta_{p, q}\mathsf{U}(1)(\mathsf{SU}(2)\times \left\{ e \right\})}. \]
    This has non-compact dual 
    \[
        \mathsf{SU}(2, 1)\times \mathsf{SL}_2(\C)/\Delta_{p, q}(\mathsf{SU}(2)\times\left\{e\right\}).
    \]
    This has isotropy representation $ \mathfrak{q}\oplus\mathfrak{p} = \qrep01\oplus\qrep12\oplus\prep01\oplus\prep14\oplus\prep22 $, where
    \resizebox{1.0 \textwidth}{!}{
            \begin{math}
            \begin{aligned}
        {\qrep01} & = \mathrm{span}\left\{ 
                 \begin{bmatrix}
                     i & 0 & 0 & 0 & 0\\
                     0 & i & 0 & 0 & 0\\
                     0 & 0 & -2i & 0 & 0\\
                     0 & 0 & 0 & 0 & 0\\
                     0 & 0 & 0 & 0 & 0
                 \end{bmatrix}
              \right\},\quad 
              {\prep01} = \mathrm{span}\left\{ 
                  \begin{bmatrix}
                    0 & 0 & 0 & 0 & 0\\
                    0 & 0 & 0 & 0 & 0\\
                    0 & 0 & 0 & 0 & 0\\
                    0 & 0 & 0 & 1 & 0\\
                    0 & 0 & 0 & 0 & -1\\
                  \end{bmatrix}
               \right\}\\
               {\qrep12} & =
               \mathrm{span}\left\{ 
                   \begin{bmatrix}
                    0 & 0 & 0 & 0 & 0\\
                    0 & 0 & 0 & 0 & 0\\
                    0 & 0 & 0 & 0 & 0\\
                    0 & 0 & 0 & 0 & 1\\
                    0 & 0 & 0 & -1 & 0\\
                   \end{bmatrix}, 
                   \begin{bmatrix}
                    0 & 0 & 0 & 0 & 0\\
                    0 & 0 & 0 & 0 & 0\\
                    0 & 0 & 0 & 0 & 0\\
                    0 & 0 & 0 & 0 & i\\
                    0 & 0 & 0 & i & 0\\
                   \end{bmatrix}
                \right\},\,
                \\
                {\prep22} & = 
                \mathrm{span}\left\{ 
                   \begin{bmatrix}
                    0 & 0 & 0 & 0 & 0\\
                    0 & 0 & 0 & 0 & 0\\
                    0 & 0 & 0 & 0 & 0\\
                    0 & 0 & 0 & 0 & 1\\
                    0 & 0 & 0 & 1 & 0\\
                   \end{bmatrix}, 
                   \begin{bmatrix}
                    0 & 0 & 0 & 0 & 0\\
                    0 & 0 & 0 & 0 & 0\\
                    0 & 0 & 0 & 0 & 0\\
                    0 & 0 & 0 & 0 & i\\
                    0 & 0 & 0 & -i & 0\\
                   \end{bmatrix}
                \right\},\\
                {\prep14} & = 
                \mathrm{span}\left\{ 
                    \begin{bmatrix}
                        0 & 0 & 1 & 0 & 0\\
                        0 & 0 & 0 & 0 & 0\\
                        1 & 0 & 0 & 0 & 0\\
                        0 & 0 & 0 & 0 & 0\\
                        0 & 0 & 0 & 0 & 0
                    \end{bmatrix},
                    \begin{bmatrix}
                        0 & 0 & i & 0 & 0\\
                        0 & 0 & 0 & 0 & 0\\
                        -i & 0 & 0 & 0 & 0\\
                        0 & 0 & 0 & 0 & 0\\
                        0 & 0 & 0 & 0 & 0
                    \end{bmatrix},
                    \begin{bmatrix}
                        0 & 0 & 0 & 0 & 0\\
                        0 & 0 & 1 & 0 & 0\\
                        0 & 1 & 0 & 0 & 0\\
                        0 & 0 & 0 & 0 & 0\\
                        0 & 0 & 0 & 0 & 0
                    \end{bmatrix},
                    \begin{bmatrix}
                        0 & 0 & 0 & 0 & 0\\
                        0 & 0 & i & 0 & 0\\
                        0 & -i & 0 & 0 & 0\\
                        0 & 0 & 0 & 0 & 0\\
                        0 & 0 & 0 & 0 & 0
                    \end{bmatrix}
                \right\},
        \end{aligned}
    \end{math}
    }
    and $ \prep22 \simeq \qrep12 $

    One other example is $ \quotient{\mathsf{SU}(3)\times \mathsf{SU}(2)\times \mathsf{SU}(2)}{\mathsf{SU}(2)\times \Delta_{p, q}\mathsf{U}(1)} $, 
    where the torus is embedded in only the second factor. If we embed the torus in only the first factor, we have a Riemannian product, so we can ignore this one. It is important however to remark that we do not know whether this space since it admits invariant Einstein metrics, since it has an $ \mathsf{SL}_{2}(\C) $ factor. Analogously, it has non-compact dual $ \mathsf{SU}(2, 1) \times \mathsf{SL}_2(\C)/\mathsf{SU}(2)\Delta_{p, q}\mathsf{U}(1) $, with isotropy representation $ \qrep01 \oplus\qrep12 \oplus \prep01 \oplus\prep14 \oplus \prep22 $, and $ \prep22 \simeq \qrep12 $.

    In order to construct any further examples, we need a 4 or 5 dimensional subalgebra of $ \mathfrak{u}(2)\oplus \mathfrak{su}(2) $. Looking at the largest proper subalgebras of each of these components, we find that there is only 1 more example. That is, we have the $ 10 $ dimensional compact homogeneous space
    \[ \quotient{\mathsf{SU}(3)\times \mathsf{SU}(2)^{2}}{\mathsf{U}(1)\Delta\mathsf{SU}(2)}, \]
    where the subalgebra $ \mathfrak{u}(1)\oplus \Delta\mathfrak{su}(2) $ has the $ \mathfrak{su}(2) $ subalgebra embedded diagonally in the $ \mathfrak{u}(2)\oplus \mathfrak{su}(2) $ isotropy subalgebra. The non-compact dual to this space is 
    \[
    	\mathsf{SU}(2, 1)\times \mathsf{SL}_2(\C)/\mathsf{U}(1)\Delta\mathsf{SU}(2).
    \]
    However, this space is not minimally presented by Proposition \ref{cor:no_sl2c_minimals}.

    \subsubsection{$ \mathsf{G}/\mathsf{K} = (\quotient{\mathsf{Sp}(2)}{\mathsf{Sp}(1)^{2}})\times (\quotient{\mathsf{SU}(2)\times \mathsf{SU}(2)}{\Delta\mathsf{SU}(2)}) $}
    We have the example $ \quotient{\mathsf{Sp}(2)\times \mathsf{SU}(2)^{2}}{\mathsf{Sp}(1)\Delta\mathsf{SU}(2)}, $
    where $ \Delta\mathsf{SU}(2) $ lives diagonally in $ \mathsf{Sp}(1) \times \mathsf{SU}(2) $, where $ \mathsf{Sp}(1) $ is the second factor in the isotropy $ \mathsf{Sp}(1)\times \mathsf{Sp}(1)\times \mathsf{SU}(2) $. There are no other homogeneous spaces for $ \dim\quotient{\mathsf{G}}{\mathsf{K}} = 7 $. This example has non-compact dual given by $ \mathsf{Sp}(1, 1)\times \mathsf{SL}_2(\C)/\mathsf{Sp}(1)\Delta\mathsf{SU}(2) $, which again is not minimally presented by Proposition \ref{cor:no_sl2c_minimals}.

    \subsubsection{$ \quotient{\mathsf{G}}{\mathsf{K}} = \quotient{\mathsf{SO}(8)}{\mathsf{SO}(7)} $}
    Since $ \mathsf{SO}(7) $ is simple, there are no codimension 1 subgroups. Furthermore, there are no codimension 2 or 3 subgroups of $ \mathsf{SO}(7) $ either, since the maximal subgroups of $ \mathsf{SO}(7) $ are $ \mathsf{SO}(6), \mathsf{SO}(2) \mathsf{SO}(5), \mathsf{SO}(3) \mathsf{SO}(4) $ and $ \mathsf{G}_{2} $ \cite{bohm_kerr_2004}.

    \subsection{$ \dim\mathsf{G}/\mathsf{K} = 8 $} We go through each symmetric space to construct the following examples.
    \subsubsection{$ \mathsf{G}/\mathsf{K} = (\quotient{\mathsf{SU}(3)}{\mathsf{U}(2)})^{2} $}
    The space $ (\quotient{\mathsf{SU}(3)}{\mathsf{SU}(2)})^{2} $ is a 10 dimensional homogeneous space with symmetric space $ (\quotient{\mathsf{SU}(3)}{\mathsf{U}(2)})^{2} $. It has non-compact dual given by $ \mathsf{SU}(2, 1)^2/\mathsf{SU}(2)^2 $. Being a homogeneous product, it has isotropy representation given by $ \mathfrak{q}\oplus \mathfrak{p} = \qrep02 \oplus \prep14 \oplus \prep24 $ \cite{arroyo_lafuente_2016}.

    Since $ \mathfrak{u}(2)\oplus \mathfrak{u}(2) $ is the direct sum of two copies of $ \mathfrak{su}(2) $ and a $ 2 $ dimensional abelian subalgebra, we may embed a torus of dimension $ 1 $ diagonally in $ \mathfrak{u}(2)\oplus \mathfrak{u}(2) $ to attain the $ 9 $ dimensional homogeneous space $ \quotient{\mathsf{SU}(3)^{2}}{\mathsf{SU}(2)^{2}\Delta_{p, q}\mathsf{U}(1)} $. It has non-compact dual given by $ \mathsf{SU}(2, 1)^2/\mathsf{SU}(2)^2\Delta_{p, q}\mathsf{U}(1) $, and the isotropy representation decomposes as $ \mathfrak{q}\oplus \mathfrak{p} = \qrep01 \oplus \prep14 \oplus \prep24 $, where
    the representations $ \prep14 $ and $ \prep24 $ on each of the $ \mathfrak{su}(2, 1) $ factors are isomorphic. By Theorem \ref{thm:nikonorov_thm_1}, this space admits no invariant Einstein metrics of negative Ricci curvature.
    
    There are no other examples of codimension $ 1 $ subalgebras in this case.
    
    \subsubsection{$ \quotient{\mathsf{G}}{\mathsf{K}} = \quotient{\mathsf{SU}(4)}{\mathsf{S}(\mathsf{U}(2)^{2})} $}
    There is only one codimension 2 subalgebra of $ \mathfrak{s}(\mathfrak{u}(2)\oplus \mathfrak{u}(2)) $, given by $ \mathfrak{s}(\mathfrak{u}(1)\oplus \mathfrak{u}(1)\oplus \mathfrak{u}(2)) $, where we take the embedding of $ \mathsf{U}(1)\mathsf{U}(1) $ inside of the first $ \mathsf{U}(2) $. This gives rise to the 10 dimensional homogeneous space $ \quotient{\mathsf{SU}(4)}{\mathsf{S}(\mathsf{U}(1)\mathsf{U}(1)\mathsf{U}(2))} $, which has non-compact dual given by $ \mathsf{SU}(2, 2)/\mathsf{S}(\mathsf{U}(1)\mathsf{U}(1)\mathsf{U}(2)) $. The isotropy representation decomposes as $ \mathfrak{q}\oplus \mathfrak{p} = \qrep12 \oplus\prep14 \oplus \prep24 $, where
    \begin{align*}
        \qrep12 &= \left\{\begin{bmatrix}
            0&z&0&0\\
            -\overline{z}&0&0&0\\
            0&0&0&0\\
            0&0&0&0
            \end{bmatrix},\, z \in \C\right\},\\
            \prep14 &= \left\{\begin{bmatrix}
            0&0&x_1&x_2\\
            0&0&0&0\\
            -\overline{x_{1}}&0&0&0\\
            -\overline{x_{2}}&0&0&0
            \end{bmatrix},\, x_i \in \C\right\},\quad \prep24 = \left\{\begin{bmatrix}
            0&0&0&0\\
            0&0&x_3&x_4\\
            0&-\overline{x_3}&0&0\\
            0&-\overline{x_{4}}&0&0
            \end{bmatrix},\, x_i \in \C\right\}.
    \end{align*}
    There is only one codimension 1 subalgebra of $ \mathfrak{s}(\mathfrak{u}(2)\oplus\mathfrak{u}(2)) $, given by $ \mathfrak{su}(2)\oplus\mathfrak{su}(2) $, which gives the homogeneous space $ \mathsf{SU}(4)/\mathsf{SU}(2)^2 $. Its non-compact dual is $ \mathsf{SU}(2, 2)/\mathsf{SU}(2)^2 $, and has isotropy representation $ \mathfrak{q}\oplus\mathfrak{p} = \qrep01 \oplus\prep14\oplus\prep14 $, so by Theorem \ref{thm:nikonorov_thm_1}, admits no invariant Einstein metrics of negative Einstein constant.

    \subsubsection{$ \quotient{\mathsf{G}}{\mathsf{K}} = \quotient{\mathsf{SU}(5)}{\mathsf{U}(4)} $}
    The only codimension 1 subalgebra of $ \mathsf{u}(4) $ is $ \mathsf{su}(4) $, giving rise to the homogeneous space $ \quotient{\mathsf{SU}(5)}{\mathsf{SU}(4)} $. It has non-compact dual $ \mathsf{SU}(4, 1)/\mathsf{SU}(4) $, and the isotropy representation decomposes as $ \mathfrak{q}\oplus \mathfrak{p} = \qrep01 \oplus \prep18 $.

    \subsubsection{$ \quotient{\mathsf{G}}{\mathsf{K}} = \quotient{\mathsf{Sp}(3)}{\mathsf{Sp}(1)\mathsf{Sp}(2)} $}
    Clearly, there are no codimension 1 or 2 subalgebras of $ \mathfrak{sp}(2) $, and the only codimension 1 or 2 subalgebra of $ \mathfrak{sp}(1) $ is $ \mathfrak{u}(1) $. Therefore, the only 9 or 10 dimensional homogeneous space is $ \quotient{\mathsf{Sp}(3)}{\mathsf{U}(1)\mathsf{Sp}(2)} $, where the embedding $ \mathsf{U}(1) \hookrightarrow \mathsf{Sp}(1) $ is the canonical one. It has non-compact dual $ \mathsf{Sp}(1, 2)/\mathsf{U}(1)\mathsf{Sp}(2) $, and the isotropy decomposes as $ \mathfrak{q}\oplus\mathfrak{p} = \qrep12 \oplus \prep28 $, where
    \begin{align*} 
        \qrep12 &= \left\{\begin{bmatrix}
        zj & 0 & 0\\
        0 & 0 & 0\\
        0 & 0 & 0
        \end{bmatrix},\, z \in \C\right\},\quad 
        \prep18 = \left\{\begin{bmatrix}
        0 & h & \ell\\
        -\overline{h} & 0 & 0\\
        -\overline{\ell} & 0 & 0
        \end{bmatrix},\, h, \ell \in \Quaternion\right\}. 
    \end{align*}

    \subsubsection{$ \quotient{\mathsf{G}}{\mathsf{K}} = \quotient{\mathsf{G}_{2}}{\mathsf{SU}(2)^{2}} $}
    There are 2 distinct embeddings of a codimension 2 subalgebra of $ \mathfrak{su}(2) \oplus \mathfrak{su}(2) $ in $ \mathsf{SU}(2)^{2} $. Their homogeneous spaces are $ \quotient{\mathsf{G}_{2}}{\mathsf{U}(2)_{1}} $ and $ \quotient{\mathsf{G}_{2}}{\mathsf{U}(2)_{3}} $, with their embeddings described in \cite{dickinson_kerr_2008}. They have non-compact duals given by $ \mathsf{G}^2_2/\mathsf{U}(2)_1 $ and $ \mathsf{G}^2_2/\mathsf{U}(2)_3 $ respectively. Moreover, their isotropy representations decompose as $ \qrep12 \oplus \prep14\oplus\prep14 $ and $ \qrep12 \oplus\prep18 $ respectively \cite[]{dickinson_kerr_2008}. There are no codimension 1 subalgebras since $ \mathfrak{su}(2)\oplus \mathfrak{su}(2) $ is semisimple.

    \subsubsection{$ \quotient{\mathsf{G}}{\mathsf{K}} = \quotient{\mathsf{SO}(9)}{\mathsf{SO}(8)} $}
    We need a codimension 1 or 2 subalgebra $ \mathfrak{h} \subseteq \mathfrak{k} $ in order to construct a 9 or 10 dimensional homogeneous space $ \quotient{\mathsf{G}}{\mathsf{H}} $. There are no codimension 1 subalgebras since $ \mathfrak{so}(8) $ is simple by the above reasoning, and there are no codimension 2 subalgebras of $ \mathfrak{so}(8) $.

    There are no other examples in $ \dim\quotient{\mathsf{G}}{\mathsf{K}} = 8 $, since the isotropy subalgebras of all other spaces admit no codimension 1 or 2 subalgebras.

    \subsection{$ \dim\mathsf{G}/\mathsf{K} = 9 $} Finally, we have the following examples from 9 dimensional symmetric spaces.
    \subsubsection{$ \quotient{\mathsf{G}}{\mathsf{K}} = \quotient{\mathsf{SU}(4)}{\mathsf{SO}(4)} $}
    Since $ \quotient{\mathsf{G}}{\mathsf{K}} $ is 9-dimensional, we need a codimension 1 subalgebra $ \mathfrak{h} \subseteq \mathfrak{k} $ in order to construct any 10-dimensional homogeneous space $ \quotient{\mathsf{G}}{\mathsf{H}} $. But since a codimension 1 subalgebra $ \mathfrak{h} $ necessarily gives rise to 1 dimensional direct sum complement $ \mathfrak{a} $, $ \mathfrak{a} $ is an abelian subalgebra of $ \mathfrak{so}(4) $. But $ \mathfrak{so}(4) $ is semisimple, so no such subalgebra can exist, so there are no 10-dimensional homogeneous spaces of the form we are looking for with associated symmetric space $ \quotient{\mathsf{SU}(4)}{\mathsf{SO}(4)} $.

    \subsubsection{$ \quotient{\mathsf{G}}{\mathsf{K}} = \quotient{\mathsf{SO}(10)}{\mathsf{SO}(9)} $}
    Since $ \quotient{\mathsf{G}}{\mathsf{K}} $ is 9-dimensional, in order to recover a $ 10 $-dimensional homogeneous space, we need a codimension 1 subalgebra $ \mathfrak{h} \subseteq \mathfrak{so}(9) $. Since no such subalgebra exists, there are no 10-dimensional homogeneous spaces of the form we are looking for with associated symmetric space $ \quotient{\mathsf{SO}(10)}{\mathsf{SO}(9)} $.

    In order to construct another homogeneous space $ \quotient{\mathsf{G}}{\mathsf{H}} $ of dimension 10, we need a subalgebra of $ \mathfrak{k} $ of codimension 1. Since all of the remaining symmetric product spaces have two factors, none of which contain a $ 1 $-dimensional torus, the largest non-product example of a homogeneous space has a codimension 2 isotropy subgroup of $ \mathsf{K} $. Hence, no further examples exist in $ \dim\quotient{\mathsf{G}}{\mathsf{K}} = 9 $.

    \section{Proof of Theorem \ref{thm:main_alek_theorem}}
    We will now prove Theorem \ref{thm:main_alek_theorem} by conducting a case-by-case analysis of the space of invariant metrics on the remaining AU-spaces. 
    \begin{proof}[Proof of Theorem \ref{thm:main_alek_theorem}]
        Without loss of generality, suppose $ \mathsf{G}/\mathsf{H} $ is an AU-space. In lieu of Theorem \ref{thm:nikonorov_thm_1} and \S\ref{sec:classification}, it is sufficient to verify the non-existence of invariant Einstein metrics with negative Einstein constant on the following cases:
        \[
        	\mathsf{SL}_2(\Quaternion)/\mathsf{Sp}(1)^2,\quad \mathsf{Sp}(1,1)/\Delta_{p, q}\mathsf{U}(1),\quad \mathsf{SL}_2(\C)\times \mathsf{SL}_2(\C)/\mathsf{U}(1)^2,    
        \]
        \[
        	\mathsf{SU}(2, 1)\times \mathsf{SL}_2(\C)/\Delta_{p, q}\mathsf{U}(1)(\mathsf{SU}(2)\times\left\{ e \right\}),\quad \mathsf{SU}(2, 1)\times \mathsf{SL}_2(\C)/\mathsf{SU}(2)\times \Delta_{p, q}\mathsf{U}(1).
        \]

        \subsection{$ \mathsf{SL}_2(\Quaternion)/\mathsf{Sp}(1)^2 $}
        We have the 1-parameter family of isomorphisms $ \psi_\lambda : \qrep14 \to \prep14 $ given by
        \[
        	\psi_\lambda\left( \begin{bmatrix}
                0 & h_1\\
                -\overline{h_1} & 0
            \end{bmatrix} \right) = 
            \begin{bmatrix}
                0 & \lambda\overline{h_1}\\
                \lambda h_1 & 0
            \end{bmatrix},\quad \lambda \in \R.
        \]
        Hence, by Schur's Lemma, every $ \ad_\mathfrak{h} $-homomorphism $ \phi : \mathfrak{q} \to \mathfrak{p} $ is given by $ \phi = \psi_\lambda $ for some $ \lambda \in \R $, and so $ \Hom(\mathfrak{q}, \mathfrak{p}) = \R $.

        Let
        \[
        	U \coloneqq \begin{bmatrix}
                1 & 0\\
                0 & 1
            \end{bmatrix},\quad D \coloneqq \begin{bmatrix}
                1 & 0\\
                0 & -1
            \end{bmatrix},\quad A \coloneqq \begin{bmatrix}
                0 & 1\\
                -1 & 0
            \end{bmatrix},\quad S \coloneqq \begin{bmatrix}
                0 & 1\\
                1 & 0
            \end{bmatrix},
        \]
        we have the ordered basis 
        \[ \mathcal{B}_{\mathfrak{sl}_2(\Quaternion)} = \left\{ D,  S, iS, jS, kS, A, iA, jA, kA,iD, iU, jD, jU, kD, kU \right\} \]
        of $ \mathfrak{sl}_2(\Quaternion) $. Using this, we attain the following ordered basis $ \mathcal{B}_{\mathfrak{q}\oplus \mathfrak{p}} = \mathcal{B}_{\prep01}\cup \mathcal{B}_{\prep14}\cup \mathcal{B}_{\qrep14} $, where the ordered bases of the irreducible submodules are
        \begin{align*}
            \mathcal{B}_{\prep01}& \coloneqq \left\{ D \right\},\quad \mathcal{B}_{\prep14} \coloneqq \left\{ S, iS, jS, kS \right\},\quad\mathcal{B}_{\qrep14}\coloneqq \left\{ A, iA, jA, kA \right\}.
        \end{align*}
        and fix an $ \mathrm{Ad}(\mathsf{Sp}(1)^2) $-invariant inner product $ \iprod{\cdot}{\cdot}_{1} $ on $ \mathfrak{p}\oplus\mathfrak{q} $ making $ \mathcal{B}_{\mathfrak{p}\oplus\mathfrak{q}} $ orthonormal. Since $ \prep14,$ and $\qrep14 $ are of real type, we have the following

        \begin{lemma}\label{lem:sl2hsp1sp1_metrics}
            Up to isometry, the space of $ \mathsf{SL}_2(\Quaternion) $-invariant metrics on $ \mathsf{SL}_2(\Quaternion)/\mathsf{Sp}(1)^2 $ is parametrised by the 4-parameter family of $ \mathrm{Ad}(\mathsf{Sp}(1)^2) $-invariant inner products on $ \mathfrak{p}\oplus \mathfrak{q} $ of the form
            \[
            	\iprod{\cdot}{\cdot}_Q = \iprod{Q\cdot}{\cdot}_1,
            \]
            where $ Q \in \mathsf{GL}_9(\R) $ is given by
            \[
            	Q = \begin{bmatrix}
                    a & 0 & 0\\
                    0 & b\mathrm{I}_4 & d\mathrm{I}_4\\
                    0 & d\mathrm{I}_4 & c\mathrm{I}_4
                \end{bmatrix},\quad a, b, c > 0,\, d^2 < bc.
            \]
        \end{lemma}
        Suppose $ \iprod{\cdot}{\cdot}_Q $ induces an $ \mathsf{SL}_2(\Quaternion) $-invariant Einstein metric on $ \mathsf{SL}_2(\Quaternion)/\mathsf{Sp}(1)^2 $ with negative Einstein constant, and denote 
        \begin{align*} 
            \mathcal{B}^Q_{\mathfrak{q}\oplus\mathfrak{p}} &= \left\{ \dfrac{1}{\sqrt{a}}D, \dfrac{1}{\sqrt{b}}S, \dfrac{1}{\sqrt{b}}iS, \dfrac{1}{\sqrt{b}}jS, \dfrac{1}{\sqrt{b}}kS, \dfrac{b}{\sqrt{\Delta}}S+\dfrac{d}{\sqrt{\Delta}}A,\right. \\
            &\hphantom{hello}\left.\dfrac{b}{\sqrt{\Delta}}iS+\dfrac{d}{\sqrt{\Delta}}iA, \dfrac{b}{\sqrt{\Delta}}jS+\dfrac{d}{\sqrt{\Delta}}jA, \dfrac{b}{\sqrt{\Delta}}kS+\dfrac{d}{\sqrt{\Delta}}kA \right\},\quad \Delta = b(bc-d^2)
        \end{align*}
        Then, $ \mathcal{B}^Q_{\mathfrak{q}\oplus \mathfrak{p}} $ is a $ g_Q $-orthonormal basis. Using \eqref{eq:ricci_formula}, the Ricci curvature takes the following form in this basis:
        \[
        	\mathrm{Ric}_Q \coloneqq \tfrac{1}{a\Delta}\begin{bmatrix}
                R_1
                 & 0 & 0\\
                0 & R_2\mathrm{I}_4 & R_3\mathrm{I}_4\\
                0 & R_3\mathrm{I}_4 & R_4\mathrm{I}_4
            \end{bmatrix},
        \]
        where
        \begin{align*}
            R_1 &= 8b(a^2 - b^2 - 2bc - c^2 + 4d^2),\\
            R_2 &= -2(a^2 b - b^3 + 7abc + bc^2 - 7ad^2 + 2bd^2 - 2cd^2),\\
            R_3 &= 2(7a-2b+2c)d,\\
            R_4 &= -2(a^2 b + b^3 + 2cd^2 - 7a(b^2 - d^2) - b(c^2 + 2d^2)).
        \end{align*}
        Since $ \iprod{\cdot}{\cdot}_Q $ is an Einstein metric, we have
        \[
        	\mathrm{Ric}_{Q}\left( \overline{e_2}, \overline{e_6} \right) = \dfrac{2(7a - 2b + 2c)d}{ab\sqrt{bc - d^2}} = 0
        \]
        and so either $ d = 0 $ or $ b = (7a + 2c)/2 $. In the former case, the Cartan decomposition is orthogonal, giving us a contradiction by Theorem \ref{thm:nikonorov_thm_1}. In the latter case, all of the off diagonal terms of the Ricci curvature vanish, and
        \[
        	\mathrm{Ric}_Q\left( \overline{e_2}, \overline{e_2} \right) = \dfrac{45a}{2(bc - d^2)},
        \]
        which is positive, a contradiction to the negativity of the Einstein constant.

        \subsection{$ \mathsf{Sp}(1,1)/\Delta_{p, q}\mathsf{U}(1) $}
        As long as $ p \neq 1 $ and $ q \neq 1 $, every $ \mathsf{Sp}(1, 1) $-invariant metric is Cartan orthogonal, and hence none of them are Einstein by Theorem \ref{thm:nikonorov_thm_1}. Unfortunately, we were not able to deal with the case $ p = q = 1 $, since the space of $ \mathsf{Sp}(1, 1) $-invariant metrics in this case is too complicated for our methods, representing one possible exception in an infinite family of cases.

        \subsection{$ \mathsf{SL}_2(\C)\times \mathsf{SL}_2(\C)/\mathsf{U}(1)^2 $}
        Using the same notation as before, let
        \[
        	U_j \coloneqq \begin{bmatrix}
                1 & 0\\
                0 & 1
            \end{bmatrix},\quad D_j\coloneqq \begin{bmatrix}
                1 & 0\\
                0 & -1
            \end{bmatrix},\quad A_j \coloneqq \begin{bmatrix}
                0 & 1\\
                -1 & 0
            \end{bmatrix},\quad S_j \coloneqq \begin{bmatrix}
                0 & 1\\
                1 & 0
            \end{bmatrix},\quad j = 1,2
        \]
        be elements in the $ j $\textsuperscript{th} copy of $ \mathfrak{sl}_2(\C) $. Then, we have the following ordered basis for $ \mathfrak{sl}_2(\C)\oplus \mathfrak{sl}_2(\C) $:
        \[
        	\mathcal{B}_{\mathfrak{sl}_2(\C)^2}\coloneqq \left\{ D_1, A_1, S_1, D_2, A_2, S_2, iU_1, iD_1, iA_1, iS_1, iU_2, iD_2, iA_2, iS_2 \right\}.
        \]

        This gives rise to the ordered basis $ \mathcal{B}_{\mathfrak{p}\oplus \mathfrak{q}}=\mathcal{B}_{\prep02}\cup \mathcal{B}_{\qrep12}\cup\mathcal{B}_{\prep12}\cup\mathcal{B}_{\qrep22}\cup\mathcal{B}_{\prep22} $ for $ \mathfrak{p} \oplus \mathfrak{q} $:
        \begin{align*}
            \mathcal{B}_{\prep02} = \left\{ 
                D_1, D_2
            \right\},&\quad
            \mathcal{B}_{\qrep12} = \left\{ 
                A_1, iS_1
            \right\},\\
            \mathcal{B}_{\qrep22} = \left\{ 
                A_2, iS_2
            \right\},&\quad
            \mathcal{B}_{\prep12} = \left\{ 
                S_1, iA_1
            \right\},\\
            \mathcal{B}_{\prep22} &= \left\{ 
                S_2, iA_2
            \right\}
        \end{align*}

        Let us fix an $ \mathrm{Ad}(\mathsf{U}(1)^2) $-invariant inner product $ \iprod{\cdot}{\cdot}_1 $ on the reductive complement of $ \mathfrak{u}(1)\oplus\mathfrak{u}(1) $ such that the ordered basis $ \mathcal{B}_{\mathfrak{p}\oplus\mathfrak{q}} =\left\{ e_i \right\}_{i = 1}^{10} = \mathcal{B}_{\prep02}\cup \mathcal{B}_{\qrep12}\cup\mathcal{B}_{\prep12}\cup\mathcal{B}_{\qrep22}\cup\mathcal{B}_{\prep22} $ for $ \mathfrak{p} \oplus \mathfrak{q} $ is orthonormal.We have the following
        \begin{lemma}\label{lem:sl2csl2cu1u1_metrics}
            Up to isometry, the space of $ \mathsf{SL}_2(\C)^2 $-invariant metrics on $ \mathsf{SL}_2(\C)^2/\mathsf{U}(1)^2 $ is parametrised by the 7-parameter family of $ \mathrm{Ad}(\mathsf{U}(1)^2) $-invariant inner products on $ \mathfrak{p}\oplus \mathfrak{q} $ of the form
            \[
            	\iprod{\cdot}{\cdot}_Q = \iprod{Q\cdot}{\cdot}_1,
            \]
            where $ Q \in \mathsf{GL}_9(\R) $ is given by
            \[
            	Q = \begin{bmatrix}
                    a & c & 0 & 0 & 0 & 0 & 0 & 0 & 0 & 0\\
                    c & b & 0 & 0 & 0 & 0 & 0 & 0 & 0 & 0\\
                    0 & 0 & d & 0 & 0 & \ell & 0 & 0 & 0 & 0\\
                    0 & 0 & 0 & d & -\ell & 0 & 0 & 0 & 0 & 0\\
                    0 & 0 & 0 & -\ell & q & 0 & 0 & 0 & 0 & 0\\
                    0 & 0 & \ell & 0 & 0 & q & 0 & 0 & 0 & 0\\
                    0 & 0 & 0 & 0 & 0 & 0 & f & 0 & 0 & n\\
                    0 & 0 & 0 & 0 & 0 & 0 & 0 & f & -n & 0\\
                    0 & 0 & 0 & 0 & 0 & 0 & 0 & -n & g & 0\\
                    0 & 0 & 0 & 0 & 0 & 0 & n & 0 & 0 & g
                \end{bmatrix},                
            \]
            where $ a, d, q, f, g > 0,\, \ell^2 < dq,\, n^2 < fg,\, c^2 < ab $.
        \end{lemma}
        \begin{proof}
            Since we have the equivalences $ \prep12 \simeq \qrep12 $ and $ \prep22 \simeq \qrep22 $ of complex type isotropy submodules, we may parametrise the space of $ \mathrm{Ad}(\mathsf{U}(1)^2) $-invariant inner products $ \overline{Q} $ in the form of the lemma, except with a block of the form $ \begin{bsmallmatrix}p & \ell\\ -\ell & p\end{bsmallmatrix} $ mapping $ \qrep12 $ to $ \prep12 $, and a block of the form $ \begin{bsmallmatrix}m & n\\ -n & m\end{bsmallmatrix} $ mapping $ \qrep22 $ to $ \prep22 $ with $ \ell^2 + p^2 < dq $ and $ m^2 + n^2 < fg $. Since $ \prep02 $ is a trivial module, we have the two one-parameter families of automorphisms on $ \qrep j2 \oplus \prep j2 $, given by $ P_j(t) = \mathrm{Ad}(\exp(tD_j)), j = 1, 2 $. In the ordered basis $ \mathcal{B}_{\qrep j2}\cup \mathcal{B}_{\prep j2} = \left\{ A_j, iS_j, S_j, iA_j \right\} $ given above, we have
            \[
            	P_j(t) = \begin{bmatrix}
                    \cosh 2t & 0 & \sinh 2t & 0\\
                    0 & \cosh 2t & 0 & \sinh 2t\\
                    \sinh 2t & 0 & \cosh 2t & 0\\
                    0 & \sinh 2t & 0 & \cosh 2t
                \end{bmatrix}.
            \]
            These give rise to an isometric family of metrics of the form $ P_j^{T}\overline{Q}P_j $ for each $ j = 1, 2 $. Moreover, in these families, $ P_j^T\overline{Q}P_j $ has the form of the Lemma in the both of the $ \qrep j2\oplus \prep j2 $ blocks if and only if there exist $ t, s\in \R $ such that
            \[
            	\tfrac12(g + f)\tanh4s + m = 0,\quad \tfrac12(d + q)\tanh4t + p = 0.
            \]
            Using the relations $ m^2 < fg $ and $ p^2 < dq $, it is easy to find $ t, s $ such that $ P_j^T \overline{Q}P_j $ satisfies the form of $ Q $ in the Lemma in both cases.
        \end{proof}
        Suppose $ Q $ gives rise to an $ \mathsf{SL}_2(\C)^2 $-invariant Einstein metric on $ \mathsf{SL}_2(\C)^2/\mathsf{U}(1)^2 $, and denote $ \mathcal{B}^Q_{\mathfrak{q}\oplus\mathfrak{p}} = \left\{ \overline{e_i} \right\} $ by
        \begin{align*}
            \mathcal{B}^Q_{\mathfrak{q}\oplus\mathfrak{p}} & = \left\{ \tfrac{1}{\sqrt{a}}D_1, \sqrt{\tfrac{a}{ab-c^2}}D_2-\tfrac{c}{\sqrt{a(ab-c^2)}}D_1, \right.\\
            &\hphantom{hello}\left. \tfrac{1}{\sqrt{d}}A_1, \tfrac{1}{\sqrt{d}}iS_1, \sqrt{\tfrac{d}{dq-\ell^2}}S_1 + \tfrac{1}{\sqrt{d(dq-\ell^2)}}iS_1, \tfrac{1}{\sqrt{d(dq-\ell^2)}}A_1-\sqrt{\tfrac{d}{dq-\ell^2}}iA_1,\right.\\
            &\hphantom{hello}\left. \tfrac{1}{\sqrt{f}}A_2, \tfrac{1}{\sqrt{f}}iS_2, \sqrt{\tfrac{f}{fg-n^2}}S_2 + \tfrac{1}{\sqrt{f(fg-n^2)}}iS_2, \tfrac{1}{\sqrt{f(fg-n^2)}}A_2-\sqrt{\tfrac{f}{fg-n^2}}iA_2 \right\}
        \end{align*}
        Then, $ \mathcal{B}^Q_{\mathfrak{q}\oplus\mathfrak{p}} $ is a $ Q $-orthonormal basis. Looking at the Ricci curvature, we have
        
        \[
        	\mathrm{Ric}_Q = \tfrac{1}{adf\Gamma\Sigma^2\Omega^2}\begin{bmatrix}
                R_1 & R_3 & 0 & 0 & 0 & 0 & 0 & 0 & 0 & 0\\
                R_3 & R_2 & 0 & 0 & 0 & 0 & 0 & 0 & 0 & 0\\
                0 & 0 & R_4 & 0 & 0 & R_5 & 0 & 0 & 0 & 0\\
                0 & 0 & 0 & R_4 & -R_5 & 0 & 0 & 0 & 0 & 0\\
                0 & 0 & 0 & -R_5 & R_6 & 0 & 0 & 0 & 0 & 0\\
                0 & 0 & R_5 & 0 & 0 & R_6 & 0 & 0 & 0 & 0\\
                0 & 0 & 0 & 0 & 0 & 0 & R_7 & 0 & 0 & R_8\\
                0 & 0 & 0 & 0 & 0 & 0 & 0 & R_7 & -R_8 & 0\\
                0 & 0 & 0 & 0 & 0 & 0 & 0 & -R_8 & R_9 & 0\\
                0 & 0 & 0 & 0 & 0 & 0 & R_8 & 0 & 0 & R_9
            \end{bmatrix},
        \]
        where $ \Gamma =  ab - c^2, \Sigma = dq - \ell^2, \Omega = fg - n^2, \sigma = dq + \ell^2, \omega = f g + n^2 $, and
        \begin{align*}
            R_1 & = -2f\Gamma\left(-2c^2d\Sigma^2\omega + \left(-2a^2d\sigma + \Sigma\left(2d^3 + 2d\left(\sigma + \Sigma\right) + q\left(\sigma +\Sigma\right)\right)\right)\Omega^2\right) \\
            R_2 & =  2cf\sqrt{\Gamma}\Sigma\left(2d\Gamma\Sigma\omega + \left(2d^3 + 2d\left(\sigma + \Sigma\right) + q\left(\sigma + \Sigma\right)\right)\right)  \\
            R_3 & = -2\left(2c^2d^3f\Sigma\Omega^2 + c^2f\Sigma\left(2d\left(\sigma - \Sigma\right) + q\left(\sigma + \Sigma\right)\right)\Omega^2\right.\\
            &\hphantom{hello}\left.+d\Sigma^2\left(2a^2f^3\Omega-2f\Gamma^2\omega+a^2g\Omega\left(\omega + \Omega\right) + 2f\Omega\left(2c^2\Omega + a^2\left(\omega + \Omega\right)\right)\right)\right)   \\
            R_4 & = -f\Sigma\left( -c^2\left( 2d^3 + q\left( \sigma - 3\Sigma \right) + 2d\left( \sigma - \Sigma \right) \right)\right.\\
            &\hphantom{hello}\left.+\Gamma\left( 2a^2 d - 2d^3 - 2d\sigma - q\sigma - 8a\Sigma + 2d\Sigma + 3q\Sigma \right) \right)\Omega^2   \\
            R_5 & = 4 f l \sqrt{\Sigma } \Omega ^2 \left(a^2 \Gamma  d-2 a \Gamma  \Sigma +\Sigma  \left(c^2+\Gamma
            \right) (d+q)\right)  \\
            R_6 & = -f\Omega\left( 2a^2d\Gamma\left( 2\sigma - \Sigma \right) + \left( c^2 + \Gamma \right)\left( 2d^3 + q\left( \sigma - 3\Sigma \right) +2d\left( \sigma - \Sigma \right) \right)\Sigma\right.\\
            &\hphantom{hello}\left. + 4a\Gamma\Sigma\left( 2d^2 - \sigma + \Sigma \right) \right)   \\
            R_7 & =  d\Sigma^2\Omega\left( -2f\Gamma\left( c^2+\Gamma \right) +a^2\left( 2f^3+g\left( \omega-3\Omega \right)+2f\left( \omega-\Omega \right) \right)+ 8a\Gamma\Omega \right)  \\
            R_8 & = -4dn\Sigma^2\sqrt{\Omega}\left( f\Gamma\left( c^2 + \Gamma \right) + a\left( a\left( f + g \right) - 2\Gamma \right)\Omega \right)  \\
            R_9 & = d\Sigma^2\left(-4f\Gamma^2\omega-\left( 2f\left( a^2f^2+4af\Gamma-\Gamma^2 \right) +a\left( a\left( 2f+g \right)-4\Gamma \right) \omega \right)\Omega\right.\\
            &\hphantom{hello}\left. +a\left( 2af+3ag-4\Gamma \right)\Omega^2 +2c^2f\Gamma\left( \Omega-2\omega \right) \right)   \\
        \end{align*}

        In the off-diagonal directions, we find that
        \begin{align*}
            \mathrm{Ric}_Q(\overline{e_1}, \overline{e_2}) & = \dfrac{2c\Delta}{a\sqrt{ab - c^2}}\\
            & \Delta = 2 + \dfrac{2\ell^2}{d^2} + \dfrac{ab - c^2}{fg - n^2} + \dfrac{2(ab - c^2)n^2}{(n^2 - fg)^2} + \dfrac{(d^2 + \ell^2)^2}{d^2(dq - \ell^2)} + \dfrac{dq - \ell^2}{d^2}.
        \end{align*}
        Since $ dq - \ell^2, fg - n^2, ab - c^2, a, b, d, q, f, g > 0 $, we have that $ \Delta > 0 $, and so $ \mathrm{Ric}_Q(\overline{e_1}, \overline{e_2}) = 0 $ if and only if $ c = 0 $, so that $ \mathsf{SL}_2(\C)^2/\mathsf{U}(1)^2 $ splits as a Riemannian product. But $ \mathsf{SL}_2(\C)/\mathsf{U}(1) $ admits no invariant Einstein metrics by \cite{arroyo_lafuente_2016}.

        \subsection{$ \mathsf{SU}(2, 1)\times \mathsf{SL}_2(\C)/\Delta_{p, q}\mathsf{U}(1)(\mathsf{SU}(2)\times\left\{ e \right\}) $}
        Looking at the decomposition of the isotropy representation, we have the following ordered bases $ \mathcal{B}_{\mathfrak{q}\oplus\mathfrak{p}} = \left\{ e_i \right\}_{i = 1}^{10} = \mathcal{B}_{\qrep01}\cup \mathcal{B}_{\prep01}\cup\mathcal{B}_{\prep22}\cup\mathcal{B}_{\qrep12}\cup\mathcal{B}_{\prep14} $ and $ \mathcal{B}_{\mathfrak{h}} = \left\{ h_i \right\}_{i = 1}^4 $:
        \resizebox{1.0 \textwidth}{!}{
            \begin{math}
            \begin{aligned}
        	\mathcal{B}_{\mathfrak{h}} &= \left\{ 
                \begin{bmatrix}
                    -ip & 0 & 0 & 0 & 0\\
                    0 & i(p + q) & 0 & 0 & 0\\
                    0 & 0 & -iq & 0 & 0\\
                    0 & 0 & 0 & i(p-q) & 0\\
                    0 & 0 & 0 & 0 & i(q-p)
                \end{bmatrix},
                \right.\\
                &\hphantom{===}\left.
                \begin{bmatrix}
                    i & 0 & 0 & 0 & 0\\
                    0 & -i & 0 & 0 & 0\\
                    0 & 0 & 0 & 0 & 0\\
                    0 & 0 & 0 & 0 & 0\\
                    0 & 0 & 0 & 0 & 0
                \end{bmatrix},
                \begin{bmatrix}
                    0 & 1 & 0 & 0 & 0\\
                    -1 & 0 & 0 & 0 & 0\\
                    0 & 0 & 0 & 0 & 0\\
                    0 & 0 & 0 & 0 & 0\\
                    0 & 0 & 0 & 0 & 0
                \end{bmatrix},
                \begin{bmatrix}
                    0 & i & 0 & 0 & 0\\
                    i & 0 & 0 & 0 & 0\\
                    0 & 0 & 0 & 0 & 0\\
                    0 & 0 & 0 & 0 & 0\\
                    0 & 0 & 0 & 0 & 0
                \end{bmatrix}
             \right\}, \\
             \mathcal{B}_{\qrep01} & = \left\{ 
                 \begin{bmatrix}
                     i & 0 & 0 & 0 & 0\\
                     0 & i & 0 & 0 & 0\\
                     0 & 0 & -2i & 0 & 0\\
                     0 & 0 & 0 & 0 & 0\\
                     0 & 0 & 0 & 0 & 0
                 \end{bmatrix}
              \right\},\quad 
              \mathcal{B}_{\prep01} = \left\{ 
                  \begin{bmatrix}
                    0 & 0 & 0 & 0 & 0\\
                    0 & 0 & 0 & 0 & 0\\
                    0 & 0 & 0 & 0 & 0\\
                    0 & 0 & 0 & 1 & 0\\
                    0 & 0 & 0 & 0 & -1\\
                  \end{bmatrix}
               \right\}\\
               \mathcal{B}_{\qrep12} & =
               \left\{ 
                   \begin{bmatrix}
                    0 & 0 & 0 & 0 & 0\\
                    0 & 0 & 0 & 0 & 0\\
                    0 & 0 & 0 & 0 & 0\\
                    0 & 0 & 0 & 0 & 1\\
                    0 & 0 & 0 & -1 & 0\\
                   \end{bmatrix}, 
                   \begin{bmatrix}
                    0 & 0 & 0 & 0 & 0\\
                    0 & 0 & 0 & 0 & 0\\
                    0 & 0 & 0 & 0 & 0\\
                    0 & 0 & 0 & 0 & i\\
                    0 & 0 & 0 & i & 0\\
                   \end{bmatrix}
                \right\},\,
                \mathcal{B}_{\prep22} = 
                \left\{ 
                   \begin{bmatrix}
                    0 & 0 & 0 & 0 & 0\\
                    0 & 0 & 0 & 0 & 0\\
                    0 & 0 & 0 & 0 & 0\\
                    0 & 0 & 0 & 0 & 1\\
                    0 & 0 & 0 & 1 & 0\\
                   \end{bmatrix}, 
                   \begin{bmatrix}
                    0 & 0 & 0 & 0 & 0\\
                    0 & 0 & 0 & 0 & 0\\
                    0 & 0 & 0 & 0 & 0\\
                    0 & 0 & 0 & 0 & i\\
                    0 & 0 & 0 & -i & 0\\
                   \end{bmatrix}
                \right\},\\
                \mathcal{B}_{\prep14} & = 
                \left\{ 
                    \begin{bmatrix}
                        0 & 0 & 1 & 0 & 0\\
                        0 & 0 & 0 & 0 & 0\\
                        1 & 0 & 0 & 0 & 0\\
                        0 & 0 & 0 & 0 & 0\\
                        0 & 0 & 0 & 0 & 0
                    \end{bmatrix},
                    \begin{bmatrix}
                        0 & 0 & i & 0 & 0\\
                        0 & 0 & 0 & 0 & 0\\
                        -i & 0 & 0 & 0 & 0\\
                        0 & 0 & 0 & 0 & 0\\
                        0 & 0 & 0 & 0 & 0
                    \end{bmatrix},
                    \begin{bmatrix}
                        0 & 0 & 0 & 0 & 0\\
                        0 & 0 & 1 & 0 & 0\\
                        0 & 1 & 0 & 0 & 0\\
                        0 & 0 & 0 & 0 & 0\\
                        0 & 0 & 0 & 0 & 0
                    \end{bmatrix},
                    \begin{bmatrix}
                        0 & 0 & 0 & 0 & 0\\
                        0 & 0 & i & 0 & 0\\
                        0 & -i & 0 & 0 & 0\\
                        0 & 0 & 0 & 0 & 0\\
                        0 & 0 & 0 & 0 & 0
                    \end{bmatrix}
                \right\}.
        \end{aligned}
        \end{math}
        }
        We have the equivalences $ \qrep12\simeq\prep22 $, so any invariant metric would make $ \prep01\oplus\qrep01 $, $ \qrep12\oplus\prep22 $ and $ \prep14 $ orthogonal. Moreover, $ \ad(e_1) $ acts trivially on $ \prep01\oplus\qrep01 $ and $ \qrep12 \oplus\prep22 $, and as $ \ad(h_1 + (p+\tfrac12q)h_2) $ on $ \prep14 $, and so acts by skew symmetric endomorphisms on $ \mathfrak{q} \oplus \mathfrak{p} $ for any invariant metric. By \cite{arroyo_lafuente_2016}*{Lemma 2.10}, for any invariant metric $ g $,
        \[
        	\mathrm{Ric}_g(e_1, e_1) = \tfrac14\sum_{i, j}g(\lb{\overline{e_i}}{\overline{e_j}},e_1)^2 \geq 0,
        \]
        where $ \left\{ \overline{e_i} \right\} $ is a $ g $-orthonormal basis.

        \subsection{$ \mathsf{SU}(2, 1)\times \mathsf{SL}_2(\C)/\mathsf{SU}(2)\times \Delta_{p, q}\mathsf{U}(1) $}
        Since the isotropy representations of this and the previous space are equivalent up to the $ \prep14 $ modules, the computation here is identical to the above.

    \end{proof}


\end{document}